\documentclass[a4paper,oneside,11pt]{article}
\usepackage{amsmath}
\usepackage{amssymb}
\usepackage[ngerman, english]{babel}
\usepackage[latin1]{inputenc}
\usepackage{amsthm}
\usepackage{MnSymbol}
\usepackage{bbm}
\usepackage{nicefrac}
\usepackage{graphicx}
\usepackage[margin = 3cm]{geometry}

\newcommand{\R}{\mathbbm{R}}
\newcommand{\N}{\mathbbm{N}}
\def\b#1{\boldsymbol{#1}}
\def\dx{\,\mathrm dx}

\def\div{\,\mathrm{div}\,}

\def\der{\mathrm{D}}

\def\epsilon{\varepsilon}

\theoremstyle{definition}
\newtheorem{defi}{Definition}
\newtheorem{bsp}{Example}
\theoremstyle{plain}
\newtheorem{lem}{Lemma}

\newtheorem{cor}{Corollary}
\newtheorem{thm}{Theorem}
\newtheorem{remark}{Remark}

\theoremstyle{remark}

\newcounter{AssCount}
\renewcommand{\theAssCount}{\textbf{(A{\arabic{AssCount}})}}
\newcounter{AssListCount}

\begin{document}

\title{Applying a phase field approach for shape optimization of a stationary Navier-Stokes flow}


\author{Harald Garcke\footnotemark[1]\and Claudia Hecht\footnotemark[1]}

\date{}

\maketitle

\renewcommand{\thefootnote}{\fnsymbol{footnote}}
%
\footnotetext[1]{Fakult\"at f\"ur Mathematik, Universit\"at Regensburg, 93040 Regensburg, Germany
({\tt \{Harald.Garcke, Claudia.Hecht\}@mathematik.uni-regensburg.de}).}
\renewcommand{\thefootnote}{\arabic{footnote}}

\begin{abstract}
\noindent We apply a phase field approach for a general shape optimization problem of a stationary Navier-Stokes flow. To be precise we add a multiple of the Ginzburg--Landau energy as a regularization to the objective functional and relax the non-permeability of the medium outside the fluid region. The resulting diffuse interface problem can be shown to be well-posed and optimality conditions are derived. We state suitable assumptions on the problem in order to derive a sharp interface limit for the minimizers and the optimality conditions. Additionally, we can derive a necessary optimality system for the sharp interface problem by geometric variations without stating additional regularity assumptions on the minimizing set.
\end{abstract}



\pagestyle{myheadings}
\markboth{H. GARCKE, C. HECHT}{SHAPE OPTIMIZATION OF NAVIER-STOKES FLOW}

\section{Introduction}
Shape optimization problems occur in many fields and industrial applications. Thus there have been a lot of contributions for this field in terms of different mathematical models, sensitivity analysis and in particular numerical methods. One main field is the structural optimization where one wants to find optimal material configurations. The second important field of shape optimization problems can be found in fluid mechanics, mainly because there are plenty of ideas, applications and contributions from industry. One typical example of such a problem is to optimize the shape of an obstacle inside a fluid in order to minimize the drag. Shape optimization problems are very challenging, in particular because the classical formulations are in general not well-posed, i.e. no minimizers exists, and it is difficult to find a stable, reliable numerical algorithm. One typical idea to overcome the first beforementioned problem is to restrict the class of possible solutions in terms of regularity or geometric constraints, see for instance \cite{bucur2006variational, kawohl2000optimal}. But this restricts the class of possible solutions and the numerical realization of those constraints is not obvious. In the field of structural optimization, another idea has been introduced, namely a regularization of the problem by adding a multiple of the perimeter of the obstacle to the objective functional. This mathematical remodelling reflects the industrial constraints of avoiding oscillations of the optimal shape on too fine scales. If one replaces the region outside the material by a so-called ersatz material, where the material properties are chosen very ``close'' to non-presence of material, the resulting problem can actually be shown to inherit a minimizer, see \cite{ambrosioButtazzo}. Bourdin and Chambolle were one of the first ones approximating the perimeter in this setting by the Ginzburg-Landau energy and hence restating this problem in a phase field setting, see \cite{bourdin_chambolle}. The resulting formulation can then be considered in standard frameworks and can also be used for numerics.\\
This idea has been applied to a fluid dynamical setting in \cite{HechtStokesEnergy, GarckeHechtStokes}, where in particular the idea of \cite{borrvall} was used in order to replace the non-fluid region by a porous medium. So far, such an idea has only been applied to the Stokes equations. But of course, in applications nonlinear fluid models and higher Reynolds numbers play an important role and so the aim of this work is to discuss the stationary state equations as state constraints. Several problems arise if we use those nonlinear equations, most of them due to the non-uniqueness of solutions to those equations. This yields that no classical control to state operator can be defined any more, and deriving optimality conditions becomes a difficult task. But also when considering the sharp interface limit, we have to identify limit elements of the fluid regions and hence need unique solvability.\\

In this work, we will discuss the following topics:
\begin{itemize}
\item In Section~\ref{s:PhasefieldProblem} we formulate a phase field porous medium formulation for shape optimization problems in a stationary Navier-Stokes flow. This will be in the following form:
\begin{align*}
& \min_{(\varphi,\b u)}\int_\Omega\frac12\alpha_\epsilon(\varphi)\left|\b u\right|^2\dx+\int_\Omega f\left(x,\b u,\der\b u\right)\dx+\gamma\int_\Omega\frac\epsilon2\left|\nabla\varphi\right|^2+\frac1\epsilon\psi(\varphi)\dx\\
&\text{subject to }\quad \alpha_\epsilon(\varphi)\b u-\mu\Delta\b u+\b u\cdot\nabla\b u=\b f,\,\,\div\b u=0\quad\text{ in }\Omega.
\end{align*}
We will in particular couple the phase field parameter $\epsilon>0$, describing the interfacial width, to the permeability of the medium outside the fluid region, given by $\left(\alpha_\epsilon(-1)\right)^{-1}\xrightarrow{\epsilon\searrow0}0$.
\item We discuss the phase field problem in terms of existence of a minimizer and necessary optimality conditions, see Section~\ref{s:PhasefieldExistence} and ~\ref{s:PhaseFieldOptCond}.
\item A corresponding perimeter penalized sharp interface problem, where the permeability of the medium outside the fluid region is zero, is formulated in Section \ref{s:SharpFormulation}. In this setting we only consider functions $\varphi$ with $\varphi\in\{-1,1\}$ a.e. and roughly outlined we solve
\begin{align*}
& \min_{(\varphi,\b u)}\int_\Omega f\left(x,\b u,\der\b u\right)\dx+\gamma c_0 P_\Omega(\{\varphi=1\})\\
&\text{subject to }\quad -\mu\Delta\b u+\b u\cdot\nabla\b u=\b f,\,\,\div\b u=0\quad\text{  in }\{\varphi=1\}.
\end{align*}

\item We derive necessary optimality conditions for the sharp interface problem under the weak regularity assumptions on the minimizing set, see Section~\ref{s:SharpOptCond}. If certain regularity of the boundary of the minimizing shape is assumed, one can restate those optimality conditions in the classical Hadamard form.
\item After formulating appropriate assumptions on the problem we can consider a sharp interface limit for the minimizers and also in the equations of the first variations, see Section \ref{s:SharpLimit}.

\end{itemize}
We want to point out that the resulting phase field problem including the porous medium approach inherits, in contrast to most formulations in shape optimization, a minimizer. Additionally, we allow a very large class of possible solutions. In particular, we do not prescribe any topological restrictions and thus we may refer to this problem also as shape and topology optimization.\\
The derived optimality conditions of the sharp interface problem generalize existing results from literature, as those can be stated with very weak regularity assumptions on the optimal shape. To calculate the geometric variation of the state variable, we actually only need the minimizing set to be Lebesgue measurable. But if appropriate regularity is assumed for the optimal shape, the stated optimality system can be shown to be equivalent to known results from literature.\\
The proposed phase field problem can also be considered in the framework of optimal control problems. This can then for instance be handled numerically by a gradient flow approach. The numerical reliability of this approach has already been examined in \cite{garckehinzeetal}.

\section{The phase field problem}\label{s:Phasefield}
\subsection{Problem formulation}\label{s:PhasefieldProblem}

In this section we will use the approach for shape optimization problems in fluids introduced in \cite{GarckeHechtStokes}, where the Stokes equations were used a a fluid model, and apply it to the stationary Navier-Stokes equations. In the following we will minimize a certain objective functional depending on the behaviour of some fluid by varying the shape, geometry and topology of the region wherein the fluid is located. The fluid region is to be chosen inside a fixed container $\Omega\subset\R^d$, which is assumed to fulfill

\begin{list}{\theAssCount}{\usecounter{AssCount}}
  \item\label{a:Omega} $\Omega\subseteq\R^d$, $d\in\{2,3\}$, is a bounded Lipschitz domain with
  outer unit normal $\b n$ such that $\R^d\setminus\overline\Omega$ is connected.
  \setcounter{AssListCount}{\value{AssCount}}
\end{list}
\begin{remark}
	The condition of $\R^d\setminus\overline\Omega$ being connected arises due to technical reasons, in particular when defining solenoidal extensions of the boundary data, see for instance Lemma~\ref{l:StatNSExistSolVF}. Anyhow, we could establish the same result for any bounded Lipschitz domain $\Omega\subset\R^d$ by using for instance a generalized version of Lemma~\ref{l:StatNSExistSolVF}, which can be found in \cite[Lemma IX.4.2]{galdi}, by including some additional conditions on the boundary data.
\end{remark}

The velocity of the fluid has prescribed Dirichlet boundary data on $\partial\Omega$, hence we may impose for instance certain in-or outflow profiles. Additionally we can assume a body force acting on the whole of $\Omega$. And so we fix for the subsequent considerations the following functions:

\begin{list}{\theAssCount}{\usecounter{AssCount}}\setcounter{AssCount}{\value{AssListCount} }
  \item\label{a:Forces}
  Let $\b f\in\b L^2(\Omega)$ denote the applied body force and
  $\b g\in\b H^{\frac12}\left(\partial\Omega\right)$ the given boundary function such that
  $\int_{\partial\Omega}\b g\cdot\b n\dx=0$.
  \setcounter{AssListCount}{\value{AssCount}}
\end{list}
We remark, that throughout this work $\R^d$-valued functions or function spaces of $\R^d$-valued functions are denoted by boldface letters.\\

The general functional to be minimized is given as $\int_\Omega f\left(x,\b u,\der\b u\right)\dx$ and hence depends on the velocity $\b u\in\b U:=\{\b v\in\b H^1(\Omega)\mid\div\b v=0,\b v|_{\partial\Omega}=\b g\}$ of the fluid and its derivative. The treatment of the pressure in the objective functional is described briefly in Section~\ref{s:Concluding}. The objective functional is chosen according to the following assumptions:

\begin{list}{\theAssCount}{\usecounter{AssCount}}\setcounter{AssCount}{\value{AssListCount}
}
 	\item\label{a:ObjectiveFctl} We choose $f:\Omega\times\R^d\times\R^{d\times d}\to\R$ as a Carath\'eodory function, thus fulfilling
 	\begin{itemize}
 	\item $f(\cdot,v,A):\Omega\to\R$ is measurable for each $v\in\R^d$, $A\in\R^{d\times d}$, and
 	\item $f(x,\cdot,\cdot):\R^d\times\R^{d\times d}\to\R$ is continuous for almost every $x\in\Omega$.
 	\end{itemize}
	Let $p\geq 2$ for $d=2$ and $2\leq p\leq 6$ for $d=3$ and assume that there exist $a\in L^1(\Omega)$, $b_1,b_2\in L^\infty(\Omega)$ such that for almost every $x\in\Omega$ it holds
 	\begin{align}\label{e:ObjFctLCaratheodory}
 	\left|f\left(x,v,A\right)\right|\leq a(x)+b_1(x)|v|^p+b_2(x)\left|A\right|^2,\quad\forall v\in\R^d, A\in\R^{d\times d}.\end{align}
 	Additionally, assume that the functional

 	\begin{align}\label{e:NotAssDefF}F:\b H^1(\Omega)\to\R,\quad F\left(\b u\right):=\int_\Omega f\left(x,\b u(x),\der\b u(x)\right)\dx\end{align}
 	is weakly lower semicontinuous and $F|_{\b U}$ is bounded from below.
\setcounter{AssListCount}{\value{AssCount}}
\end{list}

\begin{remark}
	Remark that condition $\eqref{e:ObjFctLCaratheodory}$ implies that $\b H^1(\Omega)\ni\b u\mapsto\int_\Omega f\left(x,\b u,\der\b u(x)\right)\dx$ is continuous, see \cite{showalter}.
	\end{remark}
	
The shape to be optimized is here the region filled with fluid and is described by a design function $\varphi\in L^1(\Omega)$. The fluid region then corresponds to $\{x\in\Omega\mid\varphi(x)=1\}$ and the non-fluid region is described by $\{x\in\Omega\mid\varphi(x)=-1\}$. We will formulate a diffuse interface problem, hence $\varphi$ is also allowed to take values in $(-1,1)$, which yields then an interfacial region. The thickness of the interface is dependent on the so-called phase field parameter $\epsilon>0$. We impose an additional volume constraint for the fluid region, i.e. $\int_\Omega\varphi\dx\leq\beta\left|\Omega\right|$, where $\beta\in(-1,1)$ is an arbitrary but fixed constant. Hence, the design space for the optimization problem is given by
\begin{align}\Phi_{ad}:=\left\{\varphi\in H^1(\Omega)\mid\left|\varphi\right|\leq1\text{ a.e. in }\Omega, \int_\Omega\varphi\dx\leq\beta\left|\Omega\right|\right\}.\end{align}
Sometimes, we will use the enlarged set of admissible control functions $\overline{\Phi}_{ad}$, which is given by
\begin{align}\overline{\Phi}_{ad}:=\left\{\varphi\in H^1(\Omega)\mid\left|\varphi\right|\leq1\text{ a.e. in }\Omega\right\}.\end{align}

In order to obtain a well-posed problem, we use the idea of perimeter penalization. Thus we add a multiple of the diffuse interface analogue of the perimeter functional, which is the Ginzburg-Landau energy, to the objective functional. To be precise we add
$$\gamma\int_\Omega\frac\epsilon2\left|\nabla\varphi\right|^2+\frac1\epsilon\psi\left(\varphi\right)\dx$$
where $\psi:\R\to\overline\R:=\R\cup\{+\infty\}$, given by
\begin{align*}
  \psi(\varphi):=
  \begin{cases}
  \frac12\left(1-\varphi^2\right), & \text{if }\left|\varphi\right|\leq1,\\
    +\infty,&\text{otherwise,}
  \end{cases}
\end{align*}
is the potential and $\gamma>0$ a fixed weighting parameter for this regularization. The region outside the fluid obeys the equations of flow through porous material with small permeability $\left(\overline\alpha_\epsilon\right)^{-1}\ll1$. Notice that we couple the parameter for the porous medium approach to the phase field parameter $\epsilon>0$. In the interfacial region we interpolate between the stationary Navier-Stokes equations and the porous medium equations by using an interpolation function $\alpha_\epsilon:[-1,1]\to[0,\overline\alpha_\epsilon]$ fulfilling the following assumptions:

\begin{list}{\theAssCount}{\usecounter{AssCount}}\setcounter{AssCount}{\value{AssListCount}}
  \item \label{a:Alpha}
  Let $\alpha_\epsilon:\left[-1,1\right]\to\left[0,\overline\alpha_\epsilon\right]$
  be decreasing, surjective and continuous for every $\epsilon>0$.
  
  It is required that $\overline\alpha_\epsilon>0$ is chosen such that
  $\lim_{\epsilon\searrow0}\overline\alpha_\epsilon=+\infty$ and $\alpha_\epsilon$
  converges pointwise to some function $\alpha_0:[-1,1]\to[0,+\infty]$. Additionally, we impose $\alpha_\delta(x)\geq\alpha_\epsilon(x)$
  if $\delta\leq\epsilon$ for all $x\in\left[-1,1\right]$, $\lim_{\epsilon\searrow0}\alpha_\epsilon(0)<\infty$ and a growth condition
  of the form $\overline\alpha_\epsilon=\hbox{o}\left(\epsilon^{-\frac23}\right)$.%
  
  \setcounter{AssListCount}{\value{AssCount}}
\end{list}

\begin{remark}\label{r:ConvergenceRateTwoDim}
For space dimension $d=2$ we can even choose
$\overline\alpha_\epsilon=\hbox{o}\left(\epsilon^{-\kappa}\right)$
for any $\kappa\in(0,1)$, see \cite{GarckeHechtStokes}.
\end{remark}

We introduce some notation for the nonlinear convective term arising in the stationary Navier-Stokes equations. We denote by
$$b:\b H^1(\Omega)\times\b H^1(\Omega)\times\b H^1(\Omega)\to\R$$
the following trilinear form
$$b\left(\b u,\b v,\b w\right):=\sum_{i,j=1}^d\int_\Omega u_i\partial_iv_jw_j\dx=\int_\Omega\b u\cdot\nabla\b v\cdot\b w\dx.$$
Using the restriction on the space dimension $d\in\{2,3\}$, the imbedding theorems and classical results, we see that this trilinear form fulfills the following properties:

\begin{lem}\label{l:PropertiesTrilinearForm}
 The form $b$ is well-defined and continuous in the space
$\b H^1\left(\Omega\right)\times\b H^1\left(\Omega\right)\times\b H^1_0\left(\Omega\right).$ Moreover we have:
 \begin{align}\label{e:ContinuityEstimateTrilinearForm}
 	\left|b\left(\b u,\b v,\b w\right)\right|\leq K_\Omega\left\|\nabla\b u\right\|_{\b L^2\left(\Omega\right)}\left\|\nabla\b v\right\|_{\b L^2\left(\Omega\right)}\left\|\nabla\b w\right\|_{\b L^2(\Omega)}\quad\forall\b u, \b w\in\b H^1_0(\Omega),\b v\in\b H^1(\Omega)
 \end{align}
 with $K_\Omega=\frac{2\sqrt2\left|\Omega\right|^{1/6}}{3}$ if $d=3$ and $K_\Omega=\frac{\left|\Omega\right|^{1/2}}{2}$ if $d=2$. Additionally, the following properties are satisfied:
 \begin{align}\label{e:TrilinearformLastTwoEqualZero}
 	b\left(\b u,\b v,\b v\right)=0\qquad\forall\b u\in\b H^1(\Omega),\div\b u=0,\,\b v\in\b H^1_0(\Omega),
 \end{align}
 \begin{align}\label{e:TrilinearformLastTwoSwitch}
 	b\left(\b u,\b v,\b w\right)=-b\left(\b u,\b w,\b v\right)\qquad\forall\b u\in\b H^1(\Omega),\div\b u=0,\, \b v,\b w\in\b H^1_0(\Omega).
 \end{align}

\end{lem}
\begin{proof}
	The stated continuity and estimate $\eqref{e:ContinuityEstimateTrilinearForm}$ can be found in \cite[Lemma IX.1.1]{galdi} and $\eqref{e:TrilinearformLastTwoEqualZero}-\eqref{e:TrilinearformLastTwoSwitch}$ are considered in \cite[Lemma IX.2.1]{galdi}.
\end{proof}

Besides, we have the following important continuity property:
\begin{lem}\label{l:TrilinearFormStrongCont}
Let $\left(\b u_n\right)_{n\in\N},\left(\b v_n\right)_{n\in\N}, \b u,\b v\in \b H^1\left(\Omega\right)$ be such that
$$\b u_n\rightharpoonup\b u,\quad\b v_n\rightharpoonup\b v\qquad\text{in }\b H^1(\Omega)$$
where $\b v_n|_{\partial\Omega}=\b v|_{\partial\Omega}$ for all $n\in\N$.\\
	Then
	$$\lim_{n\to\infty}b\left(\b u_n,\b v_n,\b w\right)=b\left(\b u,\b v,\b w\right)\quad\forall\b w\in\b H^1(\Omega).$$
	Moreover, one can show that
	\begin{align}\label{e:StatNSStrongCong}\b H^1(\Omega)\times\b H^1(\Omega)\ni\left(\b u,\b v\right)\mapsto b\left(\b u,\cdot,\b v\right)\in\b H^{-1}(\Omega)\end{align}
	is strongly continuous.
	

\end{lem}
\begin{proof}
	We apply the idea of \cite[Lemma 72.5]{zeidler4} and make in particular use of the compact imbedding $\b H^1(\Omega)\hookrightarrow\b L^3(\Omega)$ and the continuous imbedding $\b H^1(\Omega)\hookrightarrow\b L^6(\Omega)$.	The strong continuity stated in $\eqref{e:StatNSStrongCong}$ follows from \cite[Lemma 72.5]{zeidler4}.
	
\end{proof}

We continue with a technical lemma that will be needed quite often and is taken from \cite[Lemma IX.4.2]{galdi}.
\begin{lem}\label{l:StatNSExistSolVF}
	Let $U$ be a bounded Lipschitz domain in $\R^d$ such that $\R^d\setminus\overline U$ is connected and let $\b v_\ast\in\b H^{\frac12}\left(\partial U\right)$ satisfy
	$$\int_{\partial U}\b v_\ast\cdot\b n\dx=0$$
	where $\b n$ denotes here the outer unit normal on $U$.\\
	Then for any $\eta>0$ there exists some $\delta=\delta\left(\eta,\b v_\ast,\b n,U\right)>0$ and a vector field $\b V=\b V(\delta)$ such that
	$$\b V\in\b H^1\left(U\right),\qquad\div\b V=0,\qquad\b V=\b v_\ast\text{ on }\partial U$$
	and verifying
	\begin{align}\label{e:ExplicitSolenoidalVectorFieldTrilinearEstimate}\left|\int_U\b u\cdot\nabla\b V\cdot\b u\dx\right|\leq\eta\left\|\nabla\b u\right\|_{\b L^2(U)}^2\quad\forall\b u\in\b H^1_0(U).\end{align}
\end{lem}

We now formulate the overall optimization problem. This is given as

\begin{equation}\begin{split}\label{e:StatNSGenMinProblemFctl}\min_{\left(\varphi,\b u\right)}J_\epsilon\left(\varphi,\b u\right)&:=\frac12\int_\Omega\alpha_\epsilon\left(\varphi\right)\left|\b u\right|^2\dx+\int_\Omega f\left(x,\b u,\der\b u\right)\dx+\gamma\int_\Omega\frac\epsilon2\left|\nabla\varphi\right|^2+\frac1\epsilon\psi\left(\varphi\right)\dx\end{split}\end{equation}
subject to $\left(\varphi,\b u\right)\in \Phi_{ad}\times\b U$ and
\begin{align}\label{e:StatNSGenConstraintsWeak}\int_\Omega\alpha_\epsilon\left(\varphi\right)\b u\cdot\b v\dx+\mu\int_\Omega\nabla\b u\cdot\nabla\b v\dx+b\left(\b u,\b u,\b v\right)=\int_\Omega\b f\cdot\b v\dx\quad\forall\b v\in\b V\end{align}
where $\b V:=\{\b v\in\b H^1_0(\Omega)\mid\div\b v=0\}$. The first term which includes the interpolation function $\alpha_\epsilon$ appearing in the objective functional $\eqref{e:StatNSGenMinProblemFctl}$ penalizes too large values for $|\b u|$ outside the fluid region (hence if $\varphi=-1$). This is a result of the choice of $\alpha_\epsilon(-1)=\overline\alpha_\epsilon\gg1$. The penalization of too large values for the velocity in the porous medium is in particular important because we want in the limit $\epsilon\searrow0$ the velocity $\b u$ to vanish outside the fluid region, see Section~\ref{s:SharpFormulation}. By this we ensure to arrive in the desired black-and-white solutions.

\subsection{Existence results for the phase field problem}\label{s:PhasefieldExistence}

We will be concerned in the following with well-posedness of the constraints $\eqref{e:StatNSGenConstraintsWeak}$ and define a solution operator called $\b S_\epsilon$, see Lemma~\ref{l:StatNSExistSolOperatorSEpsilon}. Since in general we might not have a \emph{unique} solution for an arbitrary $\varphi\in\overline \Phi_{ad}$, the solution operator may be set valued, and so we cannot reformulate the problem into minimizing a reduced objective functional as it was possible in \cite{GarckeHechtStokes}.\\
Afterwards, we show existence of minimizers for the optimal control problem $\eqref{e:StatNSGenMinProblemFctl}-\eqref{e:StatNSGenConstraintsWeak}$.

\begin{lem}\label{l:StatNSExistSolOperatorSEpsilon}
	For every $\varphi\in L^1(\Omega)$ with $\left|\varphi(x)\right|\leq1$ a.e. in $\Omega$ there exists at least one $\b u\in\b U$ fulfilling $\eqref{e:StatNSGenConstraintsWeak}$.\\	
	This defines a set-valued solution operator for the constraints, which will be denoted by
	$$\b S_\epsilon\left(\varphi\right):=\left\{\b u\in\b U\mid\b u\text{ solves } \eqref{e:StatNSGenConstraintsWeak}\right\}\qquad\forall\varphi\in\overline \Phi_{ad}.$$
\end{lem}

\begin{proof}	
	For showing the existence of a velocity field $\b u\in\b U$ satisfying $\eqref{e:StatNSGenConstraintsWeak}$ we apply the arguments of \cite[Theorem 72.A]{zeidler4}, which is an application of the theory on pseudo-monotone operators. To this end, we fix $\varphi\in L^1(\Omega)$ with $\left|\varphi\right|\leq1$ a.e. in $\Omega$.\\
 		At first, we rewrite the non-homogeneous problem into a homogeneous one analogously to \cite[Theorem 1.5, Chapter II]{temam} by defining $\b\psi\in\b H^1(\Omega)$ as a solution of
		\begin{align*}
 			\div\b\psi=0\quad\text{in }\Omega,\qquad\quad	\b\psi=\b g\quad\text{on }\partial\Omega,
 		\end{align*}	
 		
 		such that 
 		\begin{align}\label{e:DiffuseExistenceSolProofPsiInequalityProperty}b\left(\b v,\b \psi,\b v\right)\leq\frac\mu2\left\|\nabla\b v\right\|_{\b L^2(\Omega)}^2\quad\forall\b v\in\b V.\end{align}
 		The existence of such a function $\b\psi$ follows from Lemma~\ref{l:StatNSExistSolVF}. Then $\b u\in\b U$ solves $\eqref{e:StatNSGenConstraintsWeak}$ if and only if $\b{\hat u}=\b u-\b\psi\in\b V$ fulfills 
 		\begin{equation}\label{e:StatNSExistenceReformulationHomogenous}\begin{split}
 			\int_\Omega\alpha_\epsilon\left(\varphi\right)\b{\hat u}\cdot\b v+\mu\nabla\b{\hat u}\cdot\nabla\b v\dx+b\left(\b{\hat u},\b{\hat u},\b v\right)+b\left(\b{\hat u},\b\psi,\b v\right)+b\left(\b\psi,\b{\hat u},\b v\right)=\left\langle\b{\hat f},\b v\right\rangle_{\b H^{-1}(\Omega)}
 		\end{split}\end{equation}
 		for all $\b v\in\b V$	where we defined $\b{\hat f}:=\b f+\mu\Delta\b\psi-\b\psi\cdot\nabla\b\psi-\alpha_\epsilon(\varphi)\b\psi\in\b H^{-1}(\Omega).$	Then we can deduce that the linear operator $A:\b V\to\b V'$, which is given by
 		$$A(\b v)(\b w):=\int_\Omega\alpha_\epsilon(\varphi)\b v\cdot\b w+\mu\nabla\b v\cdot\nabla\b w\dx+b\left(\b v,\b \psi,\b w\right)+b\left(\b\psi,\b v,\b w\right)\quad\forall\b v,\b w\in\b V,$$
 		is monotone because
 		\begin{equation}\begin{split}\label{e:StatNSExistProofMonotonEstimateForA}
 			\left\langle A\b v-A\b w,\b v-\b w\right\rangle_{\b V'}&=\underbrace{\int_\Omega\alpha_\epsilon\left(\varphi\right)\left|\b v-\b w\right|^2\dx}_{\geq0}+\mu\left\|\nabla\left(\b v-\b w\right)\right\|_{\b L^2(\Omega)}^2\underbrace{+b\left(\b v-\b w,\b \psi,\b v-\b w\right)}_{\stackrel{\eqref{e:DiffuseExistenceSolProofPsiInequalityProperty}}{\geq}-\frac\mu2\left\|\nabla\left(\b v-\b w\right)\right\|_{\b L^2(\Omega)}^2}+\\
 			&+\underbrace{b\left(\b \psi,\b v-\b w,\b v-\b w\right)}_{\stackrel{\eqref{e:TrilinearformLastTwoEqualZero}}{=}0}\geq \frac\mu2\left\|\nabla\left(\b v-\b w\right)\right\|_{\b L^2(\Omega)}^2\geq0\quad\forall\b v,\b w\in\b V.
 		\end{split}\end{equation}
 		Thus, $A$ is a monotone and linear operator and therefore pseudo-monotone. Defining $B:\b V\to\b V'$ by 
 		$$B\left(\b v\right)\left(\b w\right)=b\left(\b v,\b v,\b w\right)=-b\left(\b v,\b w,\b v\right)\quad\forall\b v,\b w\in\b V$$
 		we see that $B$ is strongly continuous (see Lemma~\ref{l:TrilinearFormStrongCont}) and thus $A+B$ is pseudo-monotone. Moreover, since both $B$ and $A$ are bounded, we get that $A+B$ is a bounded operator, and from
 		$B\left(\b v\right)\left(\b v\right)=b\left(\b v,\b v,\b v\right)=0$ and estimate $\eqref{e:StatNSExistProofMonotonEstimateForA}$ we see that $A+B:\b V\to\b V'$ is coercive.\\
 		For this reason, we can apply the main theorem on pseudo-monotone operators (see for instance \cite[27.3]{zeidler2b}) to get the existence of some $\b{\hat u}\in\b V$ such that $\eqref{e:StatNSExistenceReformulationHomogenous}$ is fulfilled, which implies that $\b u:=\b{\hat u}+\b\psi\in\b U$ fulfills $\eqref{e:StatNSGenConstraintsWeak}$.\\

\end{proof}

In general we won't have a unique solution $\b u$ of $\eqref{e:StatNSGenConstraintsWeak}$. But under an additional assumption, which will be fulfilled for example for minimizers of $\eqref{e:StatNSGenMinProblemFctl}-\eqref{e:StatNSGenConstraintsWeak}$ if $\epsilon$ is small enough, see Corollary~\ref{c:StatNSUEpsilonUniqueIfEpsilonSmall}, we can show uniqueness:

\begin{lem}\label{l:EpsilonUniqueStateEquiForSmallU}
	Assume that there exists a solution $\b u\in\b U$ of $\eqref{e:StatNSGenConstraintsWeak}$ such that it holds
	\begin{align}\label{e:NecessarySmallnessUForEpsilonUnique}\left\|\nabla\b u\right\|_{\b L^2(\Omega)}<\frac{\mu}{K_\Omega}.\end{align}
	Then this is the only solution of $\eqref{e:StatNSGenConstraintsWeak}$.
\end{lem}

\begin{proof}
 Assume $\b u\in\b U$ fulfills $\eqref{e:StatNSGenConstraintsWeak}$ and it holds $\eqref{e:NecessarySmallnessUForEpsilonUnique}$. Moreover, assume $\widehat{\b u}\in\b U$ is another solution of $\eqref{e:StatNSGenConstraintsWeak}$. 
Similar to \cite[Theorem IX.2.1]{galdi} we define $\b z:=\widehat{\b u}-\b u$ and see that $\b z$ satisfies
 $$\int_\Omega\alpha_\epsilon\left(\varphi\right)\b z\cdot\b v\dx+\mu\int_\Omega\nabla\b z\cdot\nabla\b v\dx+b\left(\widehat{\b u},\widehat{\b u},\b v\right)-b\left(\b u,\b u,\b v\right)=0\qquad\forall\b v\in\b V.$$
 Using the trilinearity of $b$ this can be rewritten as
 $$\int_\Omega\alpha_\epsilon\left(\varphi\right)\b z\cdot\b v\dx+\mu\int_\Omega\nabla\b z\cdot\nabla\b v\dx+b\left(\b z,\b z,\b v\right)+b\left(\b z,\b u,\b v\right)+b\left(\b u,\b z,\b v\right)=0\qquad\forall\b v\in\b V.$$
 Inserting $\b z\in\b V$ as a test function and using Lemma~\ref{l:PropertiesTrilinearForm} we obtain therefrom
 
 $$\int_\Omega\alpha_\epsilon\left(\varphi\right)\left|\b z\right|^2\dx+\mu\int_\Omega\left|\nabla\b z\right|^2\dx+b\left(\b z,\b u,\b z\right)=0.$$
 This gives us in view of $\alpha_\epsilon\geq 0$ and $\eqref{e:ContinuityEstimateTrilinearForm}$
 $$\mu\left\|\nabla\b z\right\|_{\b L^2(\Omega)}^2\leq K_\Omega\left\|\nabla\b z\right\|_{\b L^2(\Omega)}^2\left\|\nabla\b u\right\|_{\b L^2(\Omega)}.$$
 Finally, we see from $\eqref{e:NecessarySmallnessUForEpsilonUnique}$ that
$$\underbrace{\left(\mu-K_\Omega\left\|\nabla\b u\right\|_{\b L^2(\Omega)}\right)}_{>0}\left\|\nabla\b z\right\|_{\b L^2(\Omega)}^2\leq 0$$
which implies together with Poincar\'e's inequality $\b z\equiv\b 0$ and thus the stated uniqueness.
 
\end{proof}

Let us now analyze the overall optimization problem given by $\eqref{e:StatNSGenMinProblemFctl}$-$\eqref{e:StatNSGenConstraintsWeak}$. After having considered the state constraints, we can deduce well-posedness of the problem as the next theorem will show.

\begin{thm}\label{l:StatNSDiffuseGenerelFctExistMini}
	There exists at least one minimizer of $\eqref{e:StatNSGenMinProblemFctl}-\eqref{e:StatNSGenConstraintsWeak}$.
\end{thm}
\begin{proof}
	We start by choosing an admissible minimizing sequence $\left(\varphi_k,\b u_k\right)_{k\in\N}\subseteq \Phi_{ad}\times\b U$, which means in particular that $\b u_k\in\b S_\epsilon\left(\varphi_k\right)$. We use the state equation $\eqref{e:StatNSGenConstraintsWeak}$ to deduce a uniform bound on $\left\|\b u_k\right\|_{\b H^1(\Omega)}$ as follows:\\
		Let $\b \psi\in\b H^1(\Omega)$ be such that $\div\b\psi=0$,  $\b\psi|_{\partial\Omega}=\b g$ and $b\left(\b v,\b\psi,\b v\right)\leq\frac\mu2\left\|\nabla\b v\right\|_{\b L^2(\Omega)}^2$ for all $\b v\in\b V$, which can be chosen due to Lemma~\ref{l:StatNSExistSolVF}. Then we see that $\b{\hat u}_k:=\b u_k-\b\psi\in\b V$ is a solution to $\eqref{e:StatNSExistenceReformulationHomogenous}$ with $\varphi$ replaced by $\varphi_k$. Testing this equation with $\b v=\b{\hat u}_k$ it follows
	
	\begin{equation}\begin{split}\label{e:EpsilonMinExistUniformBoundUReformulate}
		\underbrace{\left(\alpha_\epsilon\left(\varphi_k\right)\b{\hat u}_k,\b{\hat u}_k\right)_{\b L^2(\Omega)}}_{\geq 0}+\mu\left\|\nabla\b{\hat u}_k\right\|_{\b L^2(\Omega)}^2+\underbrace{b\left(\b{\hat u}_k,\b\psi,\b{\hat u}_k\right)}_{\geq-\frac\mu2\left\|\nabla\b{\hat u}_k\right\|_{\b L^2(\Omega)}^2}=\left\langle\b{\hat f}_k,\b{\hat u}_k\right\rangle_{\b H^{-1}\left(\Omega\right)}=\\
		=\left(\b f,\b{\hat u}_k\right)_{\b L^2(\Omega)}-\mu\int_\Omega\nabla\b\psi\cdot\nabla\b{\hat u}_k\dx-b\left(\b\psi,\b\psi,\b{\hat u}_k\right)-\int_\Omega\alpha_\epsilon\left(\varphi_k\right)\b\psi\cdot\b{\hat u}_k\dx.
	\end{split}\end{equation}
	Now using the inequalities of Poincar\'e and Young we can deduce therefrom the existence of some constant $c>0$ such that
	\begin{align}\label{e:EpsilonMinExistUniformBoundUResult}\left\|\nabla\b{\hat u}_k\right\|_{\b L^2(\Omega)}^2\leq c\left(\left\|\b f\right\|_{\b L^2(\Omega)}^2+\mu\left\|\nabla\b\psi\right\|_{\b L^2(\Omega)}^2+\left\|\b\psi\right\|_{\b H^1(\Omega)}^4+\overline\alpha_\epsilon^2\left\|\b \psi\right\|_{\b L^2(\Omega)}^2\right).\end{align}
	Applying again Poincar\'e's inequality and inserting $\b u_k=\b{\hat u}_k+\b\psi$ we obtain therefrom a bound on $\left\|\b u_k\right\|_{\b H^1(\Omega)}$ uniform in $k\in\N$. Moreover, the uniform bound on $\left(J_\epsilon\left(\varphi_k,\b u_k\right)\right)_{k\in\N}$ implies that $\sup_{k\in\N}\left\|\nabla\varphi_k\right\|_{L^2(\Omega)}<\infty$. Besides $\varphi_k\in\Phi_{ad}$ for all $k\in\N$, and so $\|\varphi_k\|_{L^\infty(\Omega)}\leq1$ $\forall k\in\N$. And so we get, after possibly choosing subsequences, the following convergence results: $\b u_k\rightharpoonup\b u_0$ in $\b H^1(\Omega)$, $\varphi_k\rightharpoonup \varphi_0$ in $H^1(\Omega)$ and thus $\lim_{k\to\infty}\|\varphi_k-\varphi_0\|_{L^2(\Omega)}=0$, $\lim_{k\to\infty}\|\b u_k-\b u_0\|_{\b L^2(\Omega)}=0$ for some element $(\b u_0,\varphi_0)\in \b U\times\Phi_{ad}$. Here we used in particular that $\Phi_{ad}$ and $\b U$ are closed and convex and thus weakly closed subspaces of $H^1(\Omega)$ and $\b H^1(\Omega)$, respectively.\\
	Next we show that $\b u_0\in\b S_\epsilon\left(\varphi_0\right)$. To see this, we make use of Lebesgue's dominated convergence theorem and the pointwise convergence of the sequences $(\b u_k)_{k\in\N}$ and $(\varphi_k)_{k\in\N}$, which follows after choosing again subsequences. From this we find directly 
	$$\lim_{k\to\infty}\int_\Omega\alpha_\epsilon(\varphi_k)\b u_k\cdot\b v\dx=\int_\Omega\alpha_\epsilon(\varphi_0)\b u_0\cdot\b v\dx\quad\forall\b v\in\b V.$$
	Making use of the continuity properties of $b$ (see Lemma~\ref{l:TrilinearFormStrongCont}) we can hence take the limit $k\to\infty$ in the weak formulation of the state equation $\eqref{e:StatNSGenConstraintsWeak}$ and see that $\b u_0$ fulfills $\eqref{e:StatNSGenConstraintsWeak}$ with $\varphi$ replaced by $\varphi_0$ and thus we have shown $\b u_0\in\b S_\epsilon(\varphi_0)$.\\
	As before we can apply Lebesgue's dominated convergence theorem to deduce
	$$\lim_{k\to\infty}\int_\Omega\alpha_\epsilon(\varphi_k)|\b u_k|^2\dx=\int_\Omega\alpha_\epsilon(\varphi_0)|\b u_0|^2\dx.$$
	Using the lower semicontinuity of the objective functional we hence obtain
	$$J_\epsilon\left(\varphi_0,\b u_0\right)\leq\liminf_{k\to\infty}J_\epsilon\left(\varphi_k,\b u_k\right)$$
	which proves that $\left(\varphi_0,\b u_0\right)$ is a minimizer of $\eqref{e:StatNSGenMinProblemFctl}-\eqref{e:StatNSGenConstraintsWeak}$.

\end{proof}

\subsection{Optimality conditions in the diffuse interface setting}\label{s:PhaseFieldOptCond}

In this section we want to derive optimality conditions by geometric variations. In the end we want to obtain an optimality system for which we can consider the limit $\epsilon\searrow0$ and hope to arrive in an optimality system for the sharp interface. The corresponding optimality conditions in the sharp interface setting will be derived in Section \ref{s:SharpOptCond} and in Section \ref{s:ConvOptSys} we consider the limit process in the equations of the first variation.\\

We choose for this section $\left(\varphi_\epsilon,\b u_\epsilon\right)\in L^1(\Omega)\times\b H^1(\Omega)$ as minimizer of $\eqref{e:StatNSGenMinProblemFctl}-\eqref{e:StatNSGenConstraintsWeak}$ such that it holds
\begin{align}\label{e:StatNSOptCondSmallnessUEpsilon}
	\left\|\nabla\b u_\epsilon\right\|_{\b L^2(\Omega)}<\frac{\mu}{K_\Omega}.
\end{align}
In particular, this implies by Lemma~\ref{l:EpsilonUniqueStateEquiForSmallU} directly $\b S_\epsilon\left(\varphi_\epsilon\right)=\left\{\b u_\epsilon\right\}$. 
\begin{remark}
	We point out, that due to Corollary~\ref{c:StatNSUEpsilonUniqueIfEpsilonSmall} we obtain under certain assumptions and for $\epsilon>0$ small enough that $\eqref{e:StatNSOptCondSmallnessUEpsilon}$ is fulfilled for any minimizer $\left(\varphi_\epsilon,\b u_\epsilon\right)$ of $\eqref{e:StatNSGenMinProblemFctl}-\eqref{e:StatNSGenConstraintsWeak}$.
\end{remark}

Throughout the following section we state additionally the following assumption:

\begin{list}{\theAssCount}{\usecounter{AssCount}}\setcounter{AssCount}{\value{AssListCount}
}
 \item\label{a:Diff} 
Assume that $\alpha_\epsilon\in C^2([-1,1])$ for all $\epsilon>0$ and $\b f\in\b H^1(\Omega)$.\\
 	Assume that $x\mapsto f(x,v,A)\in\R$ is in $W^{1,1}(\Omega)$ for all $(v,A)\in\R^d\times\R^{d\times d}$ and the partial derivatives $\der_2 f\left(x,\cdot,A\right)$, $\der_3f\left(x,v,\cdot\right)$ exist for all $v\in\R^d$, $A\in\R^{d\times d}$ and a.e. $x\in\Omega$. Let $p\geq 2$ for $d=2$ and $2\leq p\leq 6$ for $d=3$ and assume that there are $\hat a\in L^1(\Omega)$, $\hat b_1,\hat b_2\in L^\infty(\Omega)$ such that for almost every $x\in\Omega$ it holds
 		\begin{align}\label{e:ObjFctLCaratheodory2}\der_{(2,3)}f\left(x,v,A\right)\leq\hat a(x)+\hat b_1(x)\left|v\right|^{p-1}+\hat b_2(x)\left|A\right|\quad\forall v\in\R^d, A\in\R^{d\times d}.\end{align}
 \setcounter{AssListCount}{\value{AssCount}}
\end{list}

	\begin{remark}
	If the objective functional fulfills Assumption~\ref{a:Diff}, we find that
$$F:\b H^1(\Omega)\ni\b u\mapsto\int_\Omega f\left(x,\b u,\der\b u\right)\dx$$
 is continuously Fr\'{e}chet differentiable and that its directional derivative is given in the following form:
 $$\der F(\b u)(\b v)=\int_\Omega\der_{(2,3)}f\left(x,\b u,\der\b u\right)\left(\b v,\der\b v\right)\dx\quad\forall \b u,\b v\in\b H^1(\Omega).$$
For details concerning Nemytskii operators we refer to \cite{showalter}. 
\end{remark}

	As we will derive first order optimality conditions by varying the domain $\Omega$ with transformations, we introduce here the admissible transformations and its corresponding velocity fields:
	
	\newcounter{VCount}
\renewcommand{\theVCount}{\textbf{(V{\arabic{VCount}})}}
\newcounter{VListCount}   
	\begin{defi}[$\mathcal V_{ad}$, $\mathcal T_{ad}$] \label{d:AdmissibleTransVel}
		 The space $\mathcal V_{ad}$ of admissible velocity fields is defined as the set of  all $V\in C\left(\left[-\tau,\tau\right]; C\left(\overline\Omega,\R^d\right)\right)$, where $\tau>0$ is some fixed, small constant, such that it holds:
		\begin{list}{\theVCount}{\usecounter{VCount}}
 \item\label{l:AssVSmooth}
 	\begin{description}\item[\textbf{(V1a)}]
		$V(t,\cdot)\in C^2\left(\overline\Omega,\R^d\right)$, \item[\hspace{0.9cm}\textbf{(V1b)}]$\exists C>0$: $\left\|V\left(\cdot,y\right)-V\left(\cdot,x\right)\right\|_{C\left(\left[-\tau,\tau\right],\R^d\right)}\leq C\left|x-y\right|\,\,\forall x,y\in\overline\Omega$,
		\end{description}
		\item\label{l:AssVNormal}
		$V(t,x)\cdot\b n(x)=0\quad\text{ on }\partial\Omega$,
		\item\label{l:AssVBoundaryG} $V(t,x)=\b0$ for a.e. $x\in\partial\Omega$ with $\b g(x)\neq \b0$.
		\setcounter{VCount}{\value{VCount}}
		\setcounter{VListCount}{\value{VCount}}
		\end{list}
		We will often use the notation $V(t)=V(t,\cdot)$.\\		
		Then the space $\mathcal T_{ad}$ of admissible transformations for the domain is defined as solutions of the ordinary differential equation
		\begin{subequations}\label{e:ODE}\begin{align}\partial_t T_t(x)=V(t,T_t(x)),\qquad T_0(x)=x\end{align}\end{subequations}
		for $V\in \mathcal V_{ad}$,
		which gives some $T:\left(-\tilde\tau,\tilde\tau\right)\times\overline\Omega\to\overline\Omega$, with $0<\tilde\tau$ small enough.
		\end{defi}
		
			\begin{remark}\label{r:NotAssPropertiesOfTransformations}
		Let $V\in{\mathcal V}_{ad}$ and $T\in{\mathcal V}_{ad}$ be the transformation associated to $V$ by $\eqref{e:ODE}$. Then $T$ admits the following properties:
		\begin{itemize}
			\item $T\left(\cdot,x\right)\in C^1\left(\left[-\tilde\tau,\tilde\tau\right],\R^d\right)$ for all $x\in\overline\Omega$,
			\item $\exists c>0, \forall x,y\in\overline\Omega$, $\left\|T\left(\cdot,x\right)-T\left(\cdot,y\right)\right\|_{C^1\left(\left[-\tilde\tau,\tilde\tau\right],\R^d\right)}\leq c\left|x-y\right|$,
			\item $\forall t\in\left[-\tilde\tau,\tilde\tau\right]$, $x\mapsto T_t(x)=T(t,x):\overline\Omega\to\overline\Omega$ is bijective,
			\item $\forall x\in\overline\Omega$, $T^{-1}(\cdot,x)\in C\left(\left[-\tilde\tau,\tilde\tau\right],\R^d\right)$,
			\item $\exists c>0,\forall x,y\in\overline\Omega,$ $\left\|T^{-1}\left(\cdot,x\right)-T^{-1}\left(\cdot,y\right)\right\|_{C\left(\left[-\tilde\tau,\tilde\tau\right],\R^d\right)}\leq c\left|x-y\right|$.
		\end{itemize}
		This is shown in \cite{delfour, deflourpaper}.
	\end{remark}

We will obtain optimality criteria by deforming the domain $\Omega$ along suitable transformations. For this purpose, we choose some $T\in{\mathcal T}_{ad}$ and denote in the following by $V\in{\mathcal V}_{ad}$ its velocity field. Let us introduce the notation
$$\varphi_\epsilon(t):=\varphi_\epsilon\circ T_t^{-1},\quad \Omega_t:=T_t(\Omega).$$
We choose elements solving the state equations corresponding to $\varphi_\epsilon(t)$:$$\b u_\epsilon(t)\in\b S_\epsilon\left(\varphi_\epsilon(t)\right).$$
This is possible since the choice of $T\in{\mathcal T}_{ad}$ implies for $\varphi_\epsilon\in\Phi_{ad}$ that $\varphi_\epsilon(t)\in\overline\Phi_{ad}$, see also Lemma~\ref{l:StatNSExistSolOperatorSEpsilon}.\\

So far, it is not clear if $\b S_\epsilon\left(\varphi_\epsilon\left(t\right)\right)=\left\{\b u_\epsilon(t)\right\}$, even though this holds true for $t=0$. But the implicit function theorem will guarantee uniqueness for small $t$, thus $\b S_\epsilon\left(\varphi_\epsilon\left(t\right)\right)=\left\{\b u_\epsilon(t)\right\}$ for $t$ small enough, and will give us at the same time differentiability of $t\mapsto\left(\b u_\epsilon(t)\circ T_t\right)$ at $t=0$, as the following lemma shows:

\begin{lem}\label{l:StatNSGenFctDiffVarParUDotpDotExist}
For $t$ small enough, we have $\b S_\epsilon\left(\varphi_\epsilon\left(t\right)\right)=\left\{\b u_\epsilon(t)\right\}$, thus the state equations $\eqref{e:StatNSGenConstraintsWeak}$ corresponding to $\varphi_\epsilon(t)$ have a \emph{unique} solution if $t$ is small enough.\\

	Moreover, we get that the mapping $\R\supset I\ni t\mapsto \b u_\epsilon(t)\circ T_t\in\b H^1(\Omega)$ is differentiable at $t=0$ (where $I$ is a small interval around 0) and $\dot{\b u}_\epsilon\left[V\right]:=\partial_t|_{t=0}\left(\b u_\epsilon(t)\circ T_t\right)$ is given as the unique weak solution to
		\begin{equation}\label{e:StatNSGenDiffVarParDotUDeltaEpsilonEquation}\begin{split}
		&\int_\Omega\alpha_\epsilon(\varphi_\epsilon)\dot{\b u}_\epsilon\left[V\right]\cdot\b z+\mu\nabla\dot{\b u}_\epsilon\left[V\right]\cdot\nabla\b z\dx+b\left(\b u_\epsilon,\dot{\b u}_\epsilon\left[V\right],\b z\right)+b\left(\dot{\b u}_\epsilon\left[V\right],\b u_\epsilon,\b z\right)=\\
		&=\int_\Omega\mu\der V(0)^T\nabla\b u_\epsilon:\nabla\b z\dx+\int_\Omega\mu\nabla\b u_\epsilon:\der V(0)^T\nabla\b z\dx+\\
		&+\int_\Omega\mu\nabla\b u_\epsilon:\nabla\left(\div V(0)\b z-\der V(0)\b z\right)\dx-\int_\Omega\mu\nabla\b u_\epsilon:\nabla\b z\div V(0)\dx+\\
		&+b\left(\der V(0)\b u_\epsilon,\b u_\epsilon,\b z\right)-b\left(\b u_\epsilon,\b u_\epsilon,\der V(0)\b z\right)+\int_\Omega\left(\nabla\b f\cdot V(0)\right)\cdot\b z\dx+\\
		&+\int_\Omega\b f\cdot\der V(0)\b z\dx-\int_\Omega\alpha_\epsilon\left(\varphi_\epsilon\right)\b u_\epsilon\cdot\der V(0)\b z\dx
	\end{split}\end{equation}
	which has to hold for all $\b z\in\b V$, together with
	\begin{align}\label{e:StatNSDiffParlVarDivDotBUEpsilonDelta}\div\dot{\b u}_\epsilon\left[V\right]=\nabla\b u_\epsilon:\der V(0).\end{align}
\end{lem}

\begin{proof}
We apply arguments similar to \cite[Theorem 2]{GarckeHechtStokes} after changing the definition of the function $F$ to

$$F:I\times\b H^1_{\b g}\left(\Omega\right)\to \b V'\times L^2_0(\Omega)$$
$$F(t,\b u):=\left(F_1(t,\b u),F_2(t,\b u)\right)\in\b V'\times L^2_0(\Omega)$$

\noindent where we define
	\begin{align*}
		F_1\left(t,\b u\right)\left(\b z\right)&:=\int_\Omega\alpha_\epsilon\left(\varphi_\epsilon\right)\b u\cdot\left(\det\der T_t^{-1}\der T_t\b z\right)\det\der T_t+\\
		&+\int_\Omega\mu\der T_t^{-T}\nabla\b u:\der T_t^{-T}\nabla\left(\det\der T_t^{-1}\der T_t\b z\right)\det\der T_t\dx+\\
		&+\int_\Omega\b u\cdot\der T_t^{-T}\nabla\b u\cdot\left(\det \der T_t^{-1}\der T_t\b z\right)\det\der T_t\dx-\\
		&-\int_\Omega\b f\circ T_t\cdot\left(\det\der T_t^{-1}\der T_t\b z\right)\det\der T_t\dx
	\end{align*}
and
	$$F_2(t,\b u)=\left(\der T_t^{-1}:\nabla\b u\right)\det\der T_t.$$

	We observe that
	$$F\left(t,\b u_\epsilon(t)\circ T_t\right)=0.$$

	Besides we find that $\der_uF\left(0,\b u_\epsilon\right)$ is for all $\b u\in\b H^1_0(\Omega)$ given by
	$$\der_uF_1\left(0,\b u_\epsilon\right)\left(\b u\right)\left(\b z\right)=\int_\Omega\alpha_\epsilon\left(\varphi_\epsilon\right)\b u\cdot\b z+\mu\nabla\b u\cdot\nabla\b z+\b u_\epsilon\cdot\nabla\b u\cdot\b z+\b u\cdot\nabla\b u_\epsilon\cdot\b z\dx\quad\forall\b z\in\b V$$
	and
	$$\der_uF_2(0,\b u_\epsilon)\b u=\div\b u.$$
	Thus we can use the solvability result for the divergence operator \cite[Lemma II.2.1.1]{sohr} and $\eqref{e:StatNSOptCondSmallnessUEpsilon}$ to obtain from Lax-Milgram's theorem that $\der_uF\left(0,\b u\right):\b H^1_0(\Omega)\to\b V'\times L^2_0(\Omega)$ is an isomorphism. As a consequence, we can apply the implicit function theorem to 
	$$G:I\times\b H^1_0(\Omega)\to\b V'\times L^2_0(\Omega),\quad G\left(t,\b v\right):=F\left(t,\b v+\b G\right),$$
	which fulfills
	$$G\left(t,\b u_\epsilon(t)\circ T_t-\b G\right)=0\quad\forall t\in I$$
	for some fixed chosen $\b G\in\b H^1(\Omega)$ such that $\b G|_{\partial\Omega}=\b g$. From this we obtain existence and uniqueness of a function $t\mapsto \b u(t)$ such that $G(t,\b u(t))=0$ for all $t\in I$ in a small interval $I$ around zero. But since $G(t,\b w_\epsilon(t)\circ T_t-\b G)=0$ for all $t\in I$ and for all $\b w_\epsilon(t)\in\b S_\epsilon\left(\varphi_\epsilon(t)\right)$, this yields already $\b u(t)=\b u_\epsilon(t)\circ T_t-\b G=\b w_\epsilon(t)\circ T_t-\b G$ for all $\b w_\epsilon(t)\in\b S_\epsilon\left(\varphi_\epsilon(t)\right)$ and thus $\b S_\epsilon\left(\varphi_\epsilon(t)\right)=\left\{\b u_\epsilon(t)\right\}$ and the first statement of the lemma follows.\\
	The implicit function theorem gives more in this setting, namely the differentiability of $t\mapsto\left(\b u_\epsilon(t)\circ T_t-\b G\right)\in\b H^1(\Omega)$ at $t=0$ and thus of $t\mapsto\left(\b u_\epsilon(t)\circ T_t\right)$ as a mapping from $I$ to $\b H^1(\Omega)$ at $t=0$ together with
	\begin{equation*}\begin{split}
 \partial_t|_{t=0}\left(\b u_\epsilon(t)\circ T_t\right)&=\partial_t|_{t=0}\left(\b u_\epsilon(t)\circ T_t-\b G\right)=-\der_uG\left(0,\b u_\epsilon-\b G\right)^{-1}\partial_t G\left(0,\b u_\epsilon-\b G\right)=\\
 &=-\der_uF\left(0,\b u_\epsilon\right)^{-1}\partial_t F\left(0,\b u_\epsilon\right)\end{split}\end{equation*}
 wherefrom we deduce the statement.	For details we refer to \cite{GarckeHechtStokes, hecht}.
\end{proof}

Using this result, we can now proceed to deriving first order optimality conditions by using the reduced functional 
$$j_\epsilon(\varphi_\epsilon(t)):=J_\epsilon(\varphi_\epsilon(t),\b S_\epsilon(\varphi_\epsilon(t)))$$
which is due to Lemma~\ref{l:StatNSGenFctDiffVarParUDotpDotExist} for $t$ small enough well-defined.\\

\begin{thm}
\label{l:StatNSDiffParVarConvOptCondDistProb}
For any minimizer $\left(\varphi_\epsilon,\b u_\epsilon\right)\in \Phi_{ad}\times\b U$ of $\eqref{e:StatNSGenMinProblemFctl}-\eqref{e:StatNSGenConstraintsWeak}$ fulfilling $\eqref{e:StatNSOptCondSmallnessUEpsilon}$ there exists some Lagrange multiplier $\lambda_\epsilon\geq0$ for the integral constraint such that the following necessary optimality conditions hold true:
\begin{align}\label{e:StatNSDiffParVarConvOptCondLimitDist}\partial_t|_{t=0}j_\epsilon\left(\varphi_\epsilon\circ T_t^{-1}\right)=-\lambda_\epsilon\int_\Omega\varphi_\epsilon\div V(0)\dx,\quad \lambda_\epsilon\left(\int_\Omega\varphi_\epsilon\dx-\beta\left|\Omega\right|\right)=0
\end{align}
for all $T\in{\mathcal T}_{ad}$ with velocity $V\in{\mathcal V}_{ad}$, where this derivative is given by the following formula:
	\begin{equation}\begin{split}\label{e:StatNSGenFctDiffVarParResVariationEpsilon}
		&\partial_t|_{t=0}j_\epsilon\left(\varphi_\epsilon\circ T_t^{-1}\right)=\int_\Omega\alpha_\epsilon\left(\varphi_\epsilon\right)\left(\b u_\epsilon\cdot\dot{\b u}_\epsilon\left[V\right]+\frac12\left|\b u_\epsilon\right|^2\div V(0)\right)\dx+\\
		&+\int_\Omega\left[\der f\left(x,\b u_\epsilon,\der\b u_\epsilon\right)\left(V(0),\dot{\b u}_\epsilon\left[V\right],\der\dot{\b u}_\epsilon\left[V\right]-\der\b u_\epsilon\der V(0)\right)+\right.\\
		&+\left.f\left(x,\b u_\epsilon,\der\b u_\epsilon\right)\div V(0)\right]\dx+\\
		&+\int_\Omega\left(\frac{\gamma\epsilon}{2}\left|\nabla\varphi_\epsilon\right|^2+\frac\gamma\epsilon\psi\left(\varphi_\epsilon\right)\right)\div V(0)-\gamma\epsilon\nabla\varphi_\epsilon\cdot\nabla V(0)\nabla\varphi_\epsilon\dx
	\end{split}\end{equation}

	and $\dot{\b u}_\epsilon\left[V\right]\in\b H^1_0(\Omega)$ is given as the solution of $\eqref{e:StatNSGenDiffVarParDotUDeltaEpsilonEquation}$-$\eqref{e:StatNSDiffParlVarDivDotBUEpsilonDelta}$.
\end{thm}
\begin{proof}
Those calculations can be carried out exactly as in \cite[Theorem 3]{GarckeHechtStokes}, where also the existence of a Lagrange multiplier  is shown.
\end{proof}

\begin{remark}\label{r:ParametricVariations}
	One can also consider the phase field problem $\eqref{e:StatNSGenMinProblemFctl}-\eqref{e:StatNSGenConstraintsWeak}$ as an optimal control problem and then derive a variational inequality by parametric variations as in standard optimal control problems, see for instance \cite{troeltzsch}. This optimality condition is then given by
	\begin{align}\label{e:VarINequlaity}\der j_\epsilon\left(\varphi_\epsilon\right)\left(\varphi-\varphi_\epsilon\right)+\lambda_\epsilon\int_\Omega\left(\varphi-\varphi_\epsilon\right)\dx\geq 0\quad\forall\varphi\in H^1(\Omega), |\varphi|\leq 1\text{ a.e. in }\Omega.\end{align}
	This criteria can also be rewritten in a more convenient adjoint formulation, compare \cite[Section 15.1]{hecht}. This approach has already been used for numerical simulations, see \cite{garckehinzeetal}, which validate the reliability of this phase field model.\\
	Assuming more regularity on $\Omega$, the boundary data $\b g$ and the objective functional one can then show, that the optimality conditions derived in Theorem~\ref{l:StatNSDiffParVarConvOptCondDistProb} are necessary for the variational inequality. To be precise, if the variational inequality is fulfilled, also $\eqref{e:StatNSDiffParVarConvOptCondLimitDist}$ is fulfilled. Roughly speaking, one can insert $\varphi\equiv \varphi_\epsilon\circ T_{-t}$ into $\eqref{e:VarINequlaity}$, divide by $t$, and use some rearrangements. For details, we refer to \cite[Section 15.3]{hecht}.	
	
\end{remark}

\section{The sharp interface problem}\label{s:SharpINterfaceProblem}
In Section~\ref{s:SharpLimit} we will consider the limit $\epsilon\searrow0$, the so-called sharp interface limit. Hence we want to send both the interface thickness and the permeability of the medium outside the fluid to zero in order to arrive in a sharp interface problem whose solutions can be considered as black-and-white solutions. This means that only pure fluid and pure non-fluid phases exist, and the permeability of the material outside the fluid is zero. In this section, we introduce and investigate the sharp interface problem that will correspond to the phase field model as $\epsilon$ tends to zero. This problem describes a general sharp interface shape and topology optimization problem in a stationary Navier-Stokes flow and is a nonlinear version of the problem description in a Stokes flow, compare \cite{GarckeHechtStokes}.

\subsection{Problem formulation}\label{s:SharpFormulation}

We start with a brief introduction in the notation of Caccioppoli sets and functions of bounded variations, but for a detailed introduction we refer to \cite{ambrosio,evans_gariepy}. We call a function $\varphi\in L^1(\Omega)$ a function of bounded variation if its distributional derivative is a vector-valued finite Radon measure. The space of functions of bounded variation in $\Omega$ is denoted by $BV(\Omega)$, and by $BV(\Omega,\{\pm1\})$ we denote functions in $BV(\Omega)$ having only the values $\pm1$ a.e. in $\Omega$. We then call a measurable set $E\subset\Omega$ Caccioppoli set if $\chi_E\in BV(\Omega)$. For any Caccioppoli set $E$, one can hence define the total variation $\left|\der\chi_E\right|(\Omega)$ of $\der\chi_E$, as $\der\chi_E$ is a finite measure. This value is then called the perimeter of $E$ in $\Omega$ and is denoted by $P_\Omega\left(E\right):=\left|\der\chi_E\right|(\Omega)$.\\

In the sharp interface problem we still define the velocity of the fluid on the whole of $\Omega$, even though there is only a part of it filled with fluid. This is realized by defining the velocity to be zero in the non-fluid region. Hence, the velocity corresponding to some design variable $\varphi\in L^1(\Omega)$ is to be chosen in the space $\b U^\varphi:=\{\b u\in\b U\mid\b u|_{\{\varphi=-1\}}=\b 0\text{ a.e. in }\Omega\}$, where we recall that the fluid regions is given by $\{\varphi=1\}$ and the non-fluid region by $\{\varphi=-1\}$. Correspondingly we define $\b V^\varphi:=\{\b u\in\b V\mid\b u|_{\{\varphi=-1\}}=\b 0\text{ a.e. in }\Omega\}$. Apparently, the space $\b U^\varphi$ may be empty since the conditions $\b u|_{\{\varphi=-1\}}=\b 0$ and $\b u|_{\partial\Omega}=\b g$ may be inconsistent with one another. As a consequence, we can only expect to find a solution of the state system if at least the solution space $\b U^\varphi$ is not empty. The design space for the sharp interface problem is given as
$$\Phi_{ad}^0:=\left\{\varphi\in BV\left(\Omega,\left\{\pm1\right\}\right)\mid\int_\Omega\varphi\dx\leq\beta\left|\Omega\right|,\,\b U^\varphi\neq\emptyset\right\}$$
and the enlarged admissible set is denoted by
$$\overline\Phi_{ad}^0:=\left\{\varphi\in BV\left(\Omega,\left\{\pm1\right\}\right)\mid\,\b U^\varphi\neq\emptyset\right\}.$$

We can then write the the sharp interface problem as

\begin{equation}\begin{split}\label{e:StokesGenMinProblemFctlSharp}\min_{\left(\varphi,\b u\right)}J_0\left(\varphi,\b u\right)&:=\int_\Omega f\left(x,\b u,\der\b u\right)\dx+\gamma c_0P_\Omega\left(\{\varphi=1\}\right)\end{split}\end{equation}
subject to $\left(\varphi,\b u\right)\in \Phi_{ad}^0\times\b U^\varphi$ and
\begin{align}\label{e:StatNSSharpConstraintsWeak}\mu\int_\Omega\nabla\b u\cdot\nabla\b v\dx+b\left(\b u,\b u,\b v\right)=\int_\Omega\b f\cdot\b v\dx\quad\forall\b v\in\b V^\varphi.\end{align}
Here, $c_0:=\int_{-1}^1\sqrt{2\psi(s)}\,\mathrm ds=\frac\pi2$ is a constant appearing due to technical reasons in the limit $\epsilon\searrow0$, compare Section~\ref{s:ConvMinim}. Recall, that $\gamma>0$ was an arbitrary weighting parameter for the perimeter penalization.

\subsection{Existence results}
Let us start by considering the state equations. Due to the nonlinearity in the equation we have to deal additionally with some technical difficulties. So we can only show the existence of a solution to $\eqref{e:StatNSSharpConstraintsWeak}$ for $\varphi\in\overline\Phi_{ad}^0$ fulfilling an additional assumption.

\begin{lem}\label{l:StatNSS0WellDefined} Let $\varphi\in L^1(\Omega)$ be such that there exists some $\b w\in\b U^\varphi$ and some $c>0$, $c<\mu$, with
\begin{align}\label{e:StatNSSOWelLDefSmallCondNec}\left|\int_\Omega\b v\cdot\nabla\b w\cdot\b v\dx\right|\leq c\left\|\nabla\b v\right\|_{\b L^2(\Omega)}^2\quad\forall\b v\in\b V^\varphi.\end{align}
Then there exists some $\b u\in\b U^\varphi$ fulfilling $\eqref{e:StatNSSharpConstraintsWeak}$. This defines a set-valued solution operator denoted by
 $$\b S_0(\varphi):=\left\{\b u\in\b U^\varphi\mid \eqref{e:StatNSSharpConstraintsWeak}\text{ is fulfilled for }\b u\right\}\quad\forall\varphi\in\overline \Phi_{ad}^0$$
 which may be empty if there is no $\b u\in\b U^\varphi$ such that $\eqref{e:StatNSSOWelLDefSmallCondNec}$ is fulfilled.
\end{lem}
\begin{remark}\label{r:StatNSS0WellDefinedMaybe}
	We point out that $\eqref{e:StatNSSOWelLDefSmallCondNec}$ is sufficient but not necessary for the existence of a solution to $\eqref{e:StatNSSharpConstraintsWeak}$, so $\b S_0(\varphi)$ may be non-empty for $\varphi\in\overline\Phi_{ad}^0$ even if $\eqref{e:StatNSSOWelLDefSmallCondNec}$ is not fulfilled.
\end{remark}

\begin{proof}

We fix some arbitrary $\varphi\in L^1(\Omega)$ with $\b U^\varphi\neq\emptyset$ and choose $\b w\in\b U^\varphi$ due to $\eqref{e:StatNSSOWelLDefSmallCondNec}$ which gives in particular a constant $0<c<\mu$ with
\begin{align}\label{e:TrilinearExistSharpCrucialWidehatEstimate}b\left(\b v,\b w,\b v\right)\leq c\left\|\nabla\b v\right\|_{\b L^2(\Omega)}^2\quad\forall\b v\in\b V^\varphi.\end{align}
Using this estimate, we can now proceed analogously to the proof of Lemma~\ref{l:StatNSExistSolOperatorSEpsilon} and use the main theorem on pseudo-monotone operators to deduce the statement. Some more details can be found in \cite[Lemma 13.1]{hecht}.
\end{proof}

Similar to the phase field setting we don't have a unique solution of the state equation~$\eqref{e:StatNSSharpConstraintsWeak}$. But under an additional constraint, which will be fulfilled for minimizers of our overall optimization problem, see Lemma~\ref{l:PropMinimizerJ0N}, we can deduce uniqueness, as the following lemma shows:

\begin{lem}\label{l:SharpUniqueStateEquiForSmallU}
	Assume that there exists a solution $\b u\in\b U^\varphi$ of $\eqref{e:StatNSSharpConstraintsWeak}$ such that it holds
	\begin{align}\label{e:NecessarySmallnessUForSharpUnique}\left\|\nabla\b u\right\|_{\b L^2(\Omega)}<\frac{\mu}{K_\Omega}.\end{align}
	Then this is the only solution of $\eqref{e:StatNSSharpConstraintsWeak}$.
\end{lem}
\begin{proof}
	 Follows as in Lemma~\ref{l:EpsilonUniqueStateEquiForSmallU}.
\end{proof}

\begin{remark}\label{r:ExistMinj0}
The existence of a minimizer for the shape optimization problem $\eqref{e:StokesGenMinProblemFctlSharp}-\eqref{e:StatNSSharpConstraintsWeak}$ may not be guaranteed in general. There are several counterexamples concerning existence of such a problem where the Laplace equation is used as a state constraint, see for instance \cite{buttazzorelaxed, dalmasomosco} and included references. But we will obtain as a consequence from our sharp interface considerations in Section~\ref{s:ConvMinim} and the fact that the porous medium -- phase field problem introduced in the previous section always admits a minimizer for each $\epsilon>0$, that under suitable assumptions also the sharp interface problem $\eqref{e:StokesGenMinProblemFctlSharp}-\eqref{e:StatNSSharpConstraintsWeak}$ has a minimizer. 
\end{remark}

\subsection{Optimality conditions}\label{s:SharpOptCond}

For this section we assume that $\left(\varphi_0,\b u_0\right)\in L^1(\Omega)\times\b H^1(\Omega)$ is a minimizer of $\eqref{e:StokesGenMinProblemFctlSharp}-\eqref{e:StatNSSharpConstraintsWeak}$ fulfilling additionally
\begin{align}\label{e:StatNSSharpParSMallnessU0Recall}
	\left\|\nabla\b u_0\right\|_{\b L^2(\Omega)}\leq\frac{\mu}{2K_\Omega}
\end{align}
and thus by Lemma~\ref{l:SharpUniqueStateEquiForSmallU} in particular $\left\{\b u_0\right\}=\b S_0(\varphi_0)$.

\begin{remark}
We will state in the next section suitable assumptions on the problem such that $\eqref{e:StatNSSharpParSMallnessU0Recall}$ is fulfilled for any minimizer $\left(\varphi_0,\b u_0\right)$ of the sharp interface problem $\eqref{e:StokesGenMinProblemFctlSharp}-\eqref{e:StatNSSharpConstraintsWeak}$, see Lemma \ref{l:PropMinimizerJ0N}. The existence of a minimizer for $\eqref{e:StokesGenMinProblemFctlSharp}-\eqref{e:StatNSSharpConstraintsWeak}$is for example guaranteed in the setting of Theorem~\ref{t:StatNSSharpConvergence}.
\end{remark}

The aim of this section is to derive first order optimality conditions for $\eqref{e:StokesGenMinProblemFctlSharp}-\eqref{e:StatNSSharpConstraintsWeak}$, thus necessary conditions that have to be fulfilled for the minimizer $\left(\varphi_0,\b u_0\right)$. Therefore we will use as in Section \ref{s:PhaseFieldOptCond} geometric variations. We point out that we do not assume any additional regularity on the minimizer. This means that our minimizing set will in general only be a Caccioppoli set. Calculating first order optimality conditions in form of geometric variations in such a general setting is to our knowledge a new result in literature.\\

We have to assume for the remainder of this section Assumption~\ref{a:Diff} to ensure differentiability of the objective functional and the external force term.

For this purpose, we fix for the rest of this subsection
$E_0:=\left\{x\in\Omega\mid\varphi_0(x)=1\right\}$. We define
$$\varphi_0(t)=\varphi_0\circ T_t^{-1},\quad \Omega_t=T_t(\Omega),\quad E_t=T_t(E_0)$$

\noindent for some given transformation $T\in{\mathcal T}_{ad}$ and see that $\varphi_0(t)\in\overline \Phi_{ad}^0$, since  the function $\left(\det\der T_t^{-1}\right)\left(\der T_t\right)\b u_0\circ T_t^{-1}\in\b U^{\varphi_0(t)}$ and so $\b U^{\varphi_0(t)}\neq\emptyset$, see also \cite[Lemma 5]{GarckeHechtStokes}.\\

We can a priori neither guarantee the existence of a solution to the state equations $\eqref{e:StatNSSharpConstraintsWeak}$ corresponding to $E_t$, nor uniqueness, even though this holds true for $E_0$. And so we start with showing the existence of a solution to the state equations corresponding to $E_t$ if $t$ is small enough:

\begin{lem}\label{l:StatNSSharpParVarExistSol}
	There exists a small interval $I\subset\R$, $0\in I$, such that there exists some $\b u_t\in\b S_0\left(\varphi_0\circ T_t^{-1}\right)$ for all $t\in I$. Moreover, there exists a constant $C>0$ independent of $t\in I$ such that it holds
	\begin{align}\label{e:StatNSSHarParVarUniformEstiamteInT}\left\|\nabla\b u_t\right\|_{\b L^2(\Omega)}\leq C.\end{align}
\end{lem}
\begin{proof}
	We define $\b u(t):=\left(\det\der T_t^{-1}\right)(\der T_t)\b u_0\circ T_t^{-1}\in\b U^{\varphi_0(t)}$ and let $\b v\in\b V$ be arbitrary. Then we have, by following the arguments of \cite[Lemma IX.1.1]{galdi}, the estimate
	\begin{align}\label{e:StatNSSharpParVarExistSOlInNeighbofhoodBEst}b\left(\b v,\b u(t),\b v\right)=-b\left(\b v,\b v,\b u(t)\right)\leq\left\|\b v\right\|_{\b L^{2d/(d-2)}(\Omega)}\left\|\nabla\b v\right\|_{\b L^2(\Omega)}\left\|\b u(t)\right\|_{\b L^d(\Omega)}.\end{align}
	Using change of variables and $\|\der T_t\|_{\infty}=\sup_{x\in\overline\Omega}\|\der T_t(x)\|_\infty\leq 1+C|t|$ and $\|\det\der T_t\|_\infty\leq 1+C|t|$, which holds for $|t|\ll1$, we find 
	\begin{align}\label{e:StatNSSharpParVarExistSOlInNeighbofhoodUtEst}
		\left\|\b u(t)\right\|_{\b L^d(\Omega)}&\leq \left(1+C|t|\right)\|\b u_0\|_{\b L^d(\Omega)}.\end{align}
Combining $\eqref{e:StatNSSharpParVarExistSOlInNeighbofhoodBEst}$ and $\eqref{e:StatNSSharpParVarExistSOlInNeighbofhoodUtEst}$ we obtain by using again estimates as in \cite[Lemma IX.1.1]{galdi} that
\begin{equation}\begin{split}\label{e:StatNSParVarEstiamteWithConstanstSmall}\left|b\left(\b v,\b u(t),\b v\right)\right|&\leq\left\|\b v\right\|_{\b L^{2d/(d-2)}(\Omega)}\left\|\nabla\b v\right\|_{\b L^2(\Omega)}\left\|\b u_0\right\|_{\b L^d(\Omega)}\left(1+C|t|\right)\leq \\
&\leq K_\Omega\left\|\nabla\b v\right\|_{\b L^2(\Omega)}^2\|\nabla\b u_0\|_{\b L^2(\Omega)}\left(1+C\left|t\right|\right)\leq\frac\mu2\left\|\nabla\b v\right\|_{\b L^2(\Omega)}^2\left(1+C|t|\right)\end{split}\end{equation}
where in the last step we made in particular use of $\eqref{e:StatNSSharpParSMallnessU0Recall}$. We hence can deduce from $\eqref{e:StatNSParVarEstiamteWithConstanstSmall}$ the existence of some interval $0\in I\subset\R$ and some constant $c>0$ with $c<\mu$ such that 
\begin{equation}\begin{split}\label{e:StatNSParVarEstiamteWithConstanstSmall2}\left|b\left(\b v,\b u(t),\b v\right)\right|\leq c\left\|\nabla\b v\right\|_{\b L^2(\Omega)}^2\quad\forall\b v\in\b V, t\in I.\end{split}\end{equation}
As by construction $\b u(t)\in\b U^{\varphi_0(t)}$ we obtain from $\eqref{e:StatNSParVarEstiamteWithConstanstSmall2}$ and Lemma~\ref{l:StatNSS0WellDefined} the existence of some $\b u_t\in\b S_0(\varphi_0(t))$ for all $t\in I$.\\

To deduce the uniform estimate $\eqref{e:StatNSSHarParVarUniformEstiamteInT}$ on $\left(\b u_t\right)_{t\in I}$ we proceed similar as in \cite[Theorem IX.2.1]{galdi} to find that $\b w_t:=\b u_t-\b u(t)\in\b V^{\varphi_0(t)}$ fulfills
\begin{align*}
\mu\left\|\nabla\b w_t\right\|_{\b L^2(\Omega)}^2+b\left(\b w_t,\b u(t),\b w_t\right)=\int_\Omega\b f\cdot\b w_t-\mu\nabla\b u(t)\cdot\nabla\b w_t\dx-b\left(\b u(t),\b u(t),\b w_t\right)
\end{align*}
and so

\begin{align*}
\mu\left\|\nabla\b w_t\right\|_{\b L^2(\Omega)}^2&\leq\left|b\left(\b w_t,\b u(t),\b w_t\right)\right|+\left\|\b f\right\|_{\b L^2(\Omega)}\left\|\b w_t\right\|_{\b L^2(\Omega)}+\mu\left\|\nabla\b u(t)\right\|_{\b L^2(\Omega)}\left\|\nabla\b w_t\right\|_{\b L^2(\Omega)}+\\
&+C\left\|\b u(t)\right\|_{\b H^1(\Omega)}^2\left\|\b w_t\right\|_{\b H^1(\Omega)}.
\end{align*}

Applying $\eqref{e:StatNSParVarEstiamteWithConstanstSmall2}$ implies then
\begin{align}\label{e:StatNSParVarEstiamteWithConstanstSmall3}
\left\|\nabla\b w_t\right\|_{\b L^2(\Omega)}^2&\leq C\left(\left\|\b u(t)\right\|_{\b H^1(\Omega)}+\left\|\b f\right\|_{\b L^2(\Omega)}^2+\left\|\nabla\b u(t)\right\|_{\b L^2(\Omega)}^2+\left\|\b u(t)\right\|_{\b H^1(\Omega)}^4\right).
\end{align}

Similar calculations as in $\eqref{e:StatNSSharpParVarExistSOlInNeighbofhoodUtEst}$ yield the existence of some $C>0$ independent of $t\in I$ such that $\sup_{t\in I}\left\|\b u(t)\right\|_{\b H^1(\Omega)}\leq C.$ And thus $\eqref{e:StatNSParVarEstiamteWithConstanstSmall3}$ implies the uniform bound $\eqref{e:StatNSSHarParVarUniformEstiamteInT}$ and we can finish the proof.
\end{proof}

In the next lemma, we will show differentiability of $t\mapsto (\b u_t\circ T_t)$ if $\b u_t\in\b S_0(\varphi_0(t))$ is a family of solutions to the state equations corresponding to the transformed state $\varphi_0(t)$. A priori, we only know existence of such a family of solutions by Lemma
~\ref{l:StatNSSharpParVarExistSol}, but we do not know if this is unique, and hence it is not clear how to choose this family. But we will obtain implicitly by the arguments of the following proof that $\b S_0(\varphi_0(t))=\{\b u_t\}$ for $|t|\ll1$ and so this choice is well-defined. One could also directly show uniqueness of a solution of the state equations corresponding to $\varphi_0(t)$ for $|t|\ll1$ by using similar arguments as in third step in the next proof, but here we deduce this fact as a consequence of the following considerations, see Corollary \ref{c:CorollaryUniqueSolSmallDeforStatNS}.

\begin{lem}\label{l:StatNSSharpParVarUTCIrcTTDiffT0}
Let $\b u_t\in\b S_0\left(\varphi_0(t)\right)$ be a family of
solutions to the state equations corresponding to $\varphi_0(t)$, whose existence is guaranteed by Lemma~\ref{l:StatNSSharpParVarExistSol} for $t\in I$, if $0\in I\subset\R$ is a small interval.\\
Then the function $I\ni t\mapsto\left(\b u_t\circ T_t\right)\in\b H^1(\Omega)$ is differentiable at $t=0$ and $\dot{\b u}_0\left[V\right]:=\partial_t|_{t=0}\left(\b u_t\circ T_t\right)\in\b H^1_0(\Omega)$ with $\dot{\b u}_0\left[V\right]|_{\{\varphi_0=-1\}}=\b 0$ is given as the unique weak solution to
	\begin{equation}\label{e:StatNSApproxParallelEquUDot}\begin{split}
	&\int_{\Omega}\mu\nabla\dot{\b u}_0\left[V\right]\cdot\nabla\b z\dx+b\left(\b u_0,\dot{\b u}_0\left[V\right],\b z\right)+b\left(\dot{\b u}_0\left[V\right],\b u_0,\b z\right)=\int_{\Omega}\mu\der V\left(0\right)^T\nabla\b u_0:\nabla\b z\dx+\\
	&+\int_{\Omega}\mu\nabla\b u_0:\der V\left(0\right)^T\nabla\b z\dx+\int_{\Omega}\mu\nabla\b u_0:\nabla\left(\div V\left(0\right)\b z-\der V\left(0\right)\b z\right)\dx-\\
	&-\int_{\Omega}\mu\nabla\b u_0:\nabla\b z\div V\left(0\right)\dx+b\left(\der V\left(0\right)\b u_0,\b u_0,\b z\right)-b\left(\b u_0,\b u_0,\der V\left(0\right)\b z\right)+\\
	&+\int_{\Omega}\left(\nabla\b f\cdot V(0)\right)\cdot\b z\dx+\int_{\Omega}\b f\cdot\der V(0)\b z\dx
\end{split}\end{equation}
 which has to hold for every $\b z\in\b V^{\varphi_0}$, together with
\begin{align}\label{e:StatNSSharpParlVarDivDotBUEpsilonDelta}\div\dot{\b u}_0\left[V\right]=\nabla\b u_0:\der V(0).\end{align}
\end{lem}
\begin{proof}
We want to use an implicit function argument similar to \cite[Theorem 6]{simonstokes}. But we cannot apply \cite[Theorem 6]{simonstokes} directly because we have nonlinear state equations and so we have to generalize this idea to this nonlinear setting here.\\
We start by defining the function $F:I\times\b V^{\varphi_0}\to \left(\b V^{\varphi_0}\right)'$
by
\begin{align*}F\left(t,\b u\right)\left(\b z\right)&=\int_{\Omega}\mu\nabla\b u:\der T_t^{-T}\nabla\left(\det\der T_t^{-1}\der T_t\b z\right)\dx-\\
&-\int_{\Omega}\mu\nabla\left(\det\der T_t\der T_t^{-1}\right)\cdot\b u:\der T_t^{-T}\nabla\left(\det\der T_t^{-1}\der T_t\b z\right)\cdot\det\der T_t\dx+\\
&+\int_{\Omega}\det\der T_t^{-1}(\der T_t)\b u\cdot\nabla\b u\left(\det\der T_t^{-1}\der T_t\b z\right)\dx-\\
&-\int_{\Omega}\der T_t\b u\cdot\nabla\left(\det\der T_t\der T_t^{-1}\right)\cdot\b u\cdot\left(\det\der T_t^{-1}\der T_t\b z\right)\dx+\\
&+\int_{\Omega}\der T_t\b u\cdot\der T_t^{-T}\nabla\b G\cdot\left(\det\der T_t^{-1}\der T_t\b z\right)\dx+\\
&+\int_{\Omega}\b G\cdot\nabla\b u\cdot\left(\det\der T_t^{-1}\der T_t\b z\right)-\b G\cdot\nabla\left(\det\der T_t\der T_t^{-1}\right)\b u\cdot\left(\der T_t\b z\right)\dx-\\
	&-\int_{\Omega}\b f\circ T_t\cdot\left(\det\der T_t^{-1}\der T_t\b z\right)\cdot\det\der T_t\dx,
	\end{align*}

\noindent where $\b G\in\b U^{\varphi_0}$ is some fixed chosen function. Roughly speaking, this means that $F\left(t,\b u\right)$ describes the state equations on $T_t(E_0)$, but transformed back to the reference region $E_0$ and reduced to homogeneous boundary data be using the function $\b G$. We will consider the state equations that are solved for the divergence-free transformation $\left(\det\der T_t\right)\left(\der T_t^{-1}\right)\b u_t\circ T_t$ of $\b u_t$ onto $T_t(E_0)$ and so there are some additional terms appearing in the definition of $F$ that correspond to $\left(\det\der T_t\right)\left(\der T_t^{-1}\right)$.\\
 Additionally, let $
	f:I\to \left(\b V^{\varphi_0}\right)'$ be defined as
\begin{align*}f\left(t\right)\left(\b z\right)&=-\int_\Omega\mu\der T_t^{-T}\nabla\b G:\der T_t^{-T}\nabla\left(\det\der T_t^{-1}\der T_t\b z\right)\cdot\det\der T_t\dx-\\
&-\int_\Omega\b G\cdot\der T_t^{-T}\nabla\b G\cdot\left(\det\der T_t^{-1}\der T_t\b z\right)\cdot\det\der T_t\dx+\\
	&+\int_\Omega\b f\circ T_t\cdot\left(\det\der T_t^{-1}\der T_t\b z\right)\cdot\det\der T_t\dx.
	\end{align*}

Direct calculations give then for all $\b u\in\b H^1(\Omega)$ and $\b z\in\b V^{\varphi_0}$ the identity
\begin{equation}\label{e:StatNSSharpPa4rVarirgendwashalt}\begin{split}
	&F\left(t,\det\der T_t(\der T_t^{-1})\b u\circ T_t-\b G\right)\left(\b z\right)+f(t)(\b z)=\\
	&=\int_{\Omega}\mu\nabla\b u\cdot\nabla\b z_t+\b u\cdot\nabla\b u\cdot\b z_t-\b f\cdot\b z_t\dx,
\end{split}\end{equation}
where we used $\b z_t:=\left(\det\der T_t\right)\der T_t^{-1}\b z\circ T_t^{-1}\in\b V^{\varphi_0(t)}$. And so in particular, this yields
\begin{align}\label{e:StatNSSHarpParVarAddLemmaZeigedasGlgerfuellt} F\left(t,\det\der T_t(\der T_t^{-1})\b u_t\circ T_t-\b G\right)=f(t)\quad\forall t\in I.\end{align}
We observe that the differentiability of $t\mapsto F\left(t,\b u\right)$ for all $\b u\in\b V^{\varphi_0}$ in a small interval around $t=0$ can be deduced directly by the regularity of the transformation $T\in{\mathcal T}_{ad}$. Moreover, we get for arbitrary $\b u\in\b V^{\varphi_0}$ and $\b z\in\b V^{\varphi_0}$:
\begin{align}\label{e:StatNSSharpParVarDerFFormel}\der_uF\left(0,\b u_0-\b G\right)\left(\b u\right)\b z=\int_\Omega\mu\nabla\b u\cdot\nabla\b z+\b u_0\cdot\nabla\b u\cdot\b z+\b u\cdot\nabla\b u_0\cdot\b z\dx.\end{align}

Now we divide the proof into several steps:
\begin{itemize}
\item \textit{1st step: } We first show that there exists some $c>0$ such that
\begin{align}\label{e:StatNSSharpParVarAddLemaFirststep}\left\|F\left(0,\b v-\b G\right)-F(0,\b u_0-\b G)\right\|_{(\b V^{\varphi_0})'}\geq c\left\|\b v-\b u_0\right\|_{\b H^1(\Omega)}\end{align}
which has to hold for all $\b v\in\b U^{\varphi_0}$.\\
To this end, we notice first that we have
\begin{equation}\begin{split}\label{e:StatNSSharpParVarAdditLemF1Ausgerechnet}
	\left(F(0,\b v-\b G)-F(0,\b u_0-\b G)\right)\b z&=\int_\Omega\mu\left(\nabla\b v-\nabla\b u_0\right)\cdot\nabla\b z+\b u\cdot\nabla\b u\cdot\b z-\\
	&-\b u_0\cdot\nabla\b u_0\cdot\b z\dx
\end{split}\end{equation}
for all $\b z\in\b V^{\varphi_0}$. Using
\begin{equation}\begin{split}\label{e:StatNSSharpParVarManipulateB}
	&b\left(\b v-\b u_0,\b v-\b u_0,\b z\right)+b\left(\b v-\b u_0,\b u_0,\b z\right)+b\left(\b u_0,\b v-\b u_0,\b z\right)=\\
	&=b\left(\b v,\b v,\b z\right)-b\left(\b u_0,\b u_0,\b z\right)
\end{split}\end{equation}
we obtain from $\eqref{e:StatNSSharpParVarAdditLemF1Ausgerechnet}$
\begin{align*}
	&\left\|F(0,\b v-\b G)-F_1(0,\b u_0-\b G)\right\|_{\left(\b V^{\varphi_0}\right)'}\geq\\
		&\geq \frac{\left|\int_\Omega\mu\left|\nabla\left(\b v-\b u_0\right)\right|^2\dx+b\left(\b v-\b u_0, \b u_0,\b v-\b u_0\right)\right|}{\left\|\b v-\b u_0\right\|_{\b H^1(\Omega)}}\geq\\
		&\geq \frac{\mu\left\|\nabla\left(\b v-\b u_0\right)\right\|_{\b L^2(\Omega)}^2-K_\Omega\left\|\nabla\b u_0\right\|_{\b L^2(\Omega)}\left\|\nabla\left(\b v-\b u_0\right)\right\|_{\b L^2(\Omega)}^2}{\left\|\b v-\b u_0\right\|_{\b H^1(\Omega)}}.
\end{align*}
As $\left\|\nabla\b u_0\right\|_{\b L^2(\Omega)}\leq\frac{\mu}{2K_\Omega}$, see $\eqref{e:StatNSSharpParSMallnessU0Recall}$, this implies the existence of a constant $c>0$ such that
\begin{align*}
	&\left\|F(0,\b v-\b G)-F(0,\b u_0-\b G)\right\|_{\left(\b V^{\varphi_0}\right)'}\geq c\frac{\left\|\nabla\left(\b v-\b u_0\right)\right\|_{\b L^2(\Omega)}^2}{\left\|\b v-\b u_0\right\|_{\b H^1(\Omega)}}\geq c\left\|\b v-\b u_0\right\|_{\b H^1(\Omega)}
\end{align*}
where we applied in the last step Poincar\'e's inequality. This proves $\eqref{e:StatNSSharpParVarAddLemaFirststep}$.

\item \textit{2nd step:} Now we want to derive a similar estimate as in the first step for the derivative of $F$. More precisely we want to show that there exists some $C>0$ such that
\begin{align}\label{e:StatNSSharpParVarAddLemResSecondStep}\left\|\der_uF\left(0,\b u_0-\b G\right)\b u\right\|_{\left(\b V^{\varphi_0}\right)'}\geq C\left\|\b u\right\|_{\b H^1(\Omega)}\quad\forall\b u\in\b V^{\varphi_0}.\end{align}
Therefore, we use the form of the derivative $\der_uF$ given by $\eqref{e:StatNSSharpParVarDerFFormel}$ and $\left\|\nabla\b u_0\right\|_{\b L^1(\Omega)}\leq\frac{\mu}{2K_\Omega}$, which follows from $\eqref{e:StatNSSharpParSMallnessU0Recall}$, and obtain similar to the first step
\begin{align*}
&\left\|\der_uF\left(0,\b u_0-\b G\right)\b u\right\|_{\left(\b V^{\varphi_0}\right)'}\geq c\frac{\left\|\nabla\b u\right\|_{\b L^2(\Omega)}^2}{\left\|\b u\right\|_{\b H^1(\Omega)}}\geq c\left\|\b u\right\|_{\b H^1(\Omega)}.
\end{align*}

\end{itemize}
For the following considerations we will use the notation
$$m(t):=\left(\det\der T_{t}\right)\left(\der T_{t}^{-1}\right)\b u_t\circ T_t\quad\forall |t|\ll1.$$
\begin{itemize}
\item \textit{3rd step:} Next we want to prove Lipschitz continuity of the mapping $I \ni t \mapsto m(t)\in\b H^1(\Omega)$ if the interval $I$ is chosen small enough.\\
We observe that the differentiability of $F$ and $f$ together with the quadratic form of $F$ imply
\begin{equation}\begin{split}\label{e:StatNSSharpParVarAddLemaDifffestimate}
	\left\|f\left(t\right)-f(0)\right\|_{\left(\b V^{\varphi_0}\right)'}\leq C\left|t\right|\quad\forall |t|\ll1
\end{split}\end{equation}
and
\begin{equation}\begin{split}\label{e:StatNSSharpParVarAddLemaDiffLargefestimate}
	&\left\|F\left(t\right)\left(\b v-\b G\right)-F(0)\left(\b v-\b G\right)\right\|_{\left(\b V^{\varphi_0}\right)'}\leq C\left|t\right|\left(\left\|\b v\right\|_{\b H^1(\Omega)}+\left\|\b v\right\|_{\b H^1(\Omega)}^2\right)\quad\forall |t|\ll1
\end{split}\end{equation}
which holds for all $\b v\in\b H^1(\Omega)$ with $\b v|_{\{\varphi_0=-1\}}=\b 0$ and $\b v|_{\partial\Omega}=\b g$. 
Moreover, it follows directly from $\eqref{e:StatNSSHarpParVarAddLemmaZeigedasGlgerfuellt}$ that
\begin{equation}\begin{split}\label{e:StatNSSharpParVarAddLemmNullerewieterungfuerLipschitz}
F\left(0,m(t)-\b G\right)&=F\left(0,m(t)-\b G\right)-F\left(t,m(t)-\b G\right)+\left(f(t)-f(0)\right)+F(0,\b u_0).
\end{split}\end{equation}
Applying $\eqref{e:StatNSSharpParVarAddLemaFirststep}$ to this identity yields

\begin{equation}\label{e:StatNSSharpParVarAddLemFirstEstimateForLipschitz}\begin{split}
 &c\left\|m(t)-\b u_0\right\|_{\b H^1(\Omega)}\leq\left\|F\left(0,m(t)-\b G\right)-F\left(0,\b u_0-\b G\right)\right\|_{\left(\b V^{\varphi_0}\right)'}=\\
&=\left\|F\left(0,m(t)-\b G\right)-F\left(t,m(t)-\b G\right)+f(t)-f(0)\right\|_{\left(\b V^{\varphi_0}\right)'}\leq \\
&\leq C\left|t\right|\left(\left\|m(t)\right\|_{\b H^1(\Omega)} +\left\|m(t)\right\|_{\b H^1(\Omega)}^2+1\right)
\end{split}\end{equation}
where we made in particular use of $\eqref{e:StatNSSharpParVarAddLemaDifffestimate}$ and $\eqref{e:StatNSSharpParVarAddLemaDiffLargefestimate}$. By using Lemma~\ref{l:StatNSSharpParVarExistSol} we can deduce that $\left\|\b u_t\circ T_t\right\|_{\b H^1(\Omega)}$ is bounded uniformly in $t$ for $|t|\ll1$ and so we can deduce from $\eqref{e:StatNSSharpParVarAddLemFirstEstimateForLipschitz}$ the existence of some $L>0$ such that it holds for $|t|\ll1$ small enough
\begin{align}\label{e:StatNSSharpParVarAddLeamLipschitz}
\left\|m(t)-m(0)\right\|_{\b H^1(\Omega)}=\left\|\left(\det\der T_t\right)\left(\der T_t^{-1}\right)\b u_t\circ T_t-\b u_0\right\|_{\b H^1(\Omega)}\leq L|t|.
\end{align}

\item \textit{4th step:} In this step we want to show the weak differentiability of $I\ni t\mapsto m(t)\in\b H^1(\Omega)$ at $t=0$. For this purpose, we start by deducing from $\eqref{e:StatNSSharpParVarAddLeamLipschitz}$ that
$$\frac{1}{|t|}\left\|m(t)-m(0)\right\|_{\b H^1(\Omega)}\leq L\quad\forall|t|\ll1.$$
And so there exists a subsequence $(t_k)_{k\in\N}$ and some element $\widetilde m\in\b V^{\varphi_0}$ such that $\left(\frac{1}{t_k}\left(m\left(t_k\right)-m(0)\right)\right)_{k\in\N}$ converges weakly in $\b H^1(\Omega)$ to $\widetilde m$. Using the differentiability assumptions on the transformation $T_t\in\overline{\mathcal T}_{ad}$, this implies additionally, that $\left(\frac{1}{t_k}\left(\b u_{t_k}\circ T_{t_k}-\b u_0\right)\right)_{k\in\N}$ converges weakly in $\b H^1(\Omega)$ to some limit element $\widetilde{\b u}\in\b H^1_0(\Omega)$ where $\widetilde{\b u}|_{\{\varphi_0=-1\}}=\b 0$.\\
As $F\left(0,\cdot\right):\b V^{\varphi_0}\to\left(\b V^{\varphi_0}\right)'$ is Fr\'echet differentiable we find that there exists some $r_F$ such that it holds for all $\b v_1,\b v_2\in\b H^1_0(\Omega)$
\begin{align}\label{e:StatNSSHarpParVarResttermDiffAdLem}\lim_{\|\b v_1-\b v_2\|_{\b H^1(\Omega)}\to0}\frac{\|r_F(\b v_1)\|_{(\b V^{\varphi_0})'}}{\|\b v_1-\b v_2\|_{\b H^1(\Omega)}}=0\end{align}
and $\der_uF\left(0,\b v_2\right)\left(\b v_1-\b v_2\right)=F\left(0,\b v_1\right)-F\left(0,\b v_2\right)+r_F(\b v_1)$. From this, we find that
\begin{equation}\label{e:StatNSSharpParVarAddLeamTermeDazuUndWeg}\begin{split}
&\der_uF\left(0,\b u_0-\b G\right)\left(m(t_k)-m(0)\right)=\left(F\left(t_k,m(0)-\b G\right)-F\left(t_k,m(t_k)-\b G\right)\right)-\\
&-\left(F\left(0,m(0)-\b G\right)-F\left(0,m(t_k)-\b G\right)\right)+\\
&+\left(F\left(0,m(0)-\b G\right)-F\left(t_k,m(0)-\b G\right)+\der_{t}F\left(0,m(0)-\b G\right)t_k\right)+\\
&+\left(f(t_k)-f(0)-f'(0)t_k\right)+f'(0)t_k-\der_{t}F\left(0,\b u_0-\b G\right)t_k+r_F\left(m(t_k)-\b G\right).
\end{split}\end{equation}

Using $\eqref{e:StatNSSharpPa4rVarirgendwashalt}$ and $\eqref{e:StatNSSharpParVarManipulateB}$ while making in particular use of the quadratic form of $F$ we can establish similar to $\eqref{e:StatNSSharpParVarAddLemaDiffLargefestimate}$
\begin{align*}
&\left\|\left(F\left(0,m\left(t_k\right)-\b G\right)-F\left(0,m(0)-\b G\right)\right)-\right.\\
&\left.-\left(F\left(t_k,m\left(t_k\right)-\b G\right)-F\left(t_k,m(0)-\b G\right)\right)\right\|_{\left(\b V^{\varphi_0}\right)'}\leq \\
&\leq C|t_k|\left(\left\|m(t_k)-m(0)\right\|_{\b H^1(\Omega)}+\left\|m\left(t_k\right)-m(0)\right\|_{\b H^1(\Omega)}^2\right)\leq C|t_k|^2\quad\forall k\gg1,
\end{align*}
where the last inequality follows from the Lipschitz continuity $\eqref{e:StatNSSharpParVarAddLeamLipschitz}$. This leads to
\begin{equation}\label{e:StatNSSHarpParVarAddLemaFirstTermWeakConv}\begin{split}&\lim_{k\to\infty}\frac{1}{|t_k|}\left\|\left(F\left(0,m(t_k)-\b G\right)-F\left(0,m(0)-\b G\right)\right)-\right.\\
&\left.-\left(F\left(t_k,m(t_k)-\b G\right)-F\left(t_k,m(0)-\b G\right)\right)\right\|_{\left(\b V^{\varphi_0}\right)'}=0.\end{split}\end{equation}
Since $F\left(\cdot,\b u_0-\b G\right):I\to\left(\b V^{\varphi_0}\right)'$ is Fr\'echet differentiable at $t=0$ we find moreover
\begin{align*}
\left\|F\left(0,\b u_0-\b G\right)-F\left(t_k,\b u_0-\b G\right)+\der_tF\left(0,\b u_0-\b G\right)t_k\right\|_{\left(\b V^{\varphi_0}\right)'}=\hbox{o}\left(|t_k|\right)
\end{align*}
and hence
\begin{align}\label{e:StatNSSHarpParVarAddLemaSecondTermWeakConv}
\lim_{k\to\infty}\left\|\frac{1}{t_k}\left(F\left(0,\b u_0-\b G\right)-F\left(t_k,\b u_0-\b G\right)+\der_tF\left(0,\b u_0-\b G\right)t_k\right)\right\|_{\left(\b V^{\varphi_0}\right)'}=0.
\end{align}
Similarly, we derive from the Fr\'echet differentiability of $f$ at $t=0$ that it holds
\begin{align}\label{e:StatNSSHarpParVarAddLemaThirdTermWeakConv}
\lim_{k\to\infty}\left\|\frac{1}{t_k}\left(f(t_k)-f(0)-f'(0)t_k\right)\right\|_{\left(\b V^{\varphi_0}\right)'}=0.
\end{align}
Now we combine $\eqref{e:StatNSSharpParVarAddLeamLipschitz}$, $\eqref{e:StatNSSHarpParVarResttermDiffAdLem}$ with the estimates $\eqref{e:StatNSSHarpParVarAddLemaFirstTermWeakConv}$, $\eqref{e:StatNSSHarpParVarAddLemaSecondTermWeakConv}$ and $\eqref{e:StatNSSHarpParVarAddLemaThirdTermWeakConv}$ to deduce from $\eqref{e:StatNSSharpParVarAddLeamTermeDazuUndWeg}$ that the weak limit $\widetilde m$ of $\left(\frac{1}{t_k}\left(m(t_k)-m(0)\right)\right)_{k\in\N}$ fulfills
\begin{align}\label{e:StatNSSharpParVarAddLemaWidetildeMSolves}
\der_uF\left(0,\b u_0-\b G\right)\widetilde m=f'(0)-\der_tF\left(0,\b u_0-\b G\right).
\end{align}
Direct calculations imply hence that $\widetilde{\b u}\in\b H^1_0(\Omega)$ with $\widetilde{\b u}|_{\{\varphi_0=-1\}}=\b 0$ solves $\eqref{e:StatNSApproxParallelEquUDot}$-$\eqref{e:StatNSSharpParlVarDivDotBUEpsilonDelta}$ and hereby we guarantee in particular solvability of $\eqref{e:StatNSApproxParallelEquUDot}$-$\eqref{e:StatNSSharpParlVarDivDotBUEpsilonDelta}$.\\

In view of the result from the second step in this proof, we find that there exists at most one solution to $\eqref{e:StatNSApproxParallelEquUDot}$-$\eqref{e:StatNSSharpParlVarDivDotBUEpsilonDelta}$, and hence $\widetilde{\b u}$ is the unique solution of $\eqref{e:StatNSApproxParallelEquUDot}$-$\eqref{e:StatNSSharpParlVarDivDotBUEpsilonDelta}$ as stated in the claim of this lemma.\\
By carrying out the same arguments for any subsequence $(t_k)_{k\in\N}$ we can conclude that $\left(\frac1t\left(m(t)-m(0)\right)\right)_t$ itself converges weakly in $\b H^1(\Omega)$ to $\widetilde m$.

\item \textit{5th step:} We now want to conclude the differentiability of $I\ni t\mapsto\b u_t\circ T_t\in\b H^1(\Omega)$ at $t=0$, which is equivalent to the differentiability of $I\ni t\mapsto m(t)\in\b H^1(\Omega)$ at $t=0$. Therefore, we have to show the strong convergence 
\begin{equation}\label{e:StatNSSHarpParVarAddLemaShowSTrongto}\lim_{t\to0}\left\|\frac1t\left(m(t)-m(0)\right)-\widetilde m\right\|_{\b H^1(\Omega)}=0.\end{equation}

For this purpose, we start by applying estimate $\eqref{e:StatNSSharpParVarAddLemResSecondStep}$, which was established in the second step of this proof, and see
\begin{equation}\label{e:StatNSSHarpParVarAddLemaShowSTrongEstim}\begin{split}
	&\left\|m(t)-m(0)-t\tilde m\right\|_{\b H^1(\Omega)}\leq C\left\|\der_uF\left(0,\b u_0-\b G\right)\left(m(t)-m(0)-t\tilde m\right)\right\|_{\left(\b V^{\varphi_0}\right)'}=\\
	&=C\left\|\der_uF\left(0,\b u_0-\b G\right)\left(m(t)-m(0)\right)-t\left(f'(0)-\der_tF\left(0,\b u_0-\b G\right)\right)\right\|_{\left(\b V^{\varphi_0}\right)'}
\end{split}\end{equation}
where we made in the last step use of $\eqref{e:StatNSSharpParVarAddLemaWidetildeMSolves}$. The considerations of the fourth step of this proof give us
\begin{align*}
\left\|\der_uF\left(0,\b u_0-\b G\right)\left(m(t)-m(0)\right)-t\left(f'(0)-\der_tF\left(0,\b u_0-\b G\right)\right)\right\|_{\left(\b V^{\varphi_0}\right)'}=\hbox{o}(|t|)
\end{align*}
and hence we find from $\eqref{e:StatNSSHarpParVarAddLemaShowSTrongEstim}$ directly $\eqref{e:StatNSSHarpParVarAddLemaShowSTrongto}$. This finally proves the statement of the lemma.

\end{itemize}

\end{proof}
From the previous lemma we obtain directly the following result concerning uniqueness of the state equations:
\begin{cor}\label{c:CorollaryUniqueSolSmallDeforStatNS}
	There exists a small interval $I\subset\R$, $0\in I$, such that $\b S_0\left(\varphi_0\circ T_t^{-1}\right)=\{\b u_t\}$ for all $t\in I$, thus there exists a unique solution to the state equations $\eqref{e:StatNSSharpConstraintsWeak}$ corresponding to small deformations $\varphi_0\circ T_t^{-1}$, $|t|\ll1$, of the minimizer $\varphi_0$.
\end{cor}
\begin{proof}
	By Lemma~\ref{l:StatNSSharpParVarExistSol} we have for every $t\in I$, if $I\subset\R$ is chosen small enough, a solution $\b u_t\in\b S_0\left(\varphi_0\circ T_t^{-1}\right)$ for the state equations $\eqref{e:StatNSSharpConstraintsWeak}$ corresponding to $\varphi_0\circ T_t^{-1}$. Lemma~\ref{l:StatNSSharpParVarUTCIrcTTDiffT0} guarantees additionally that $t\mapsto \left(\b u_t\circ T_t\right)\in\b H^1(\Omega)$ is continuous. Hence there exists some $t'>0$ such that
	$$\left\|\nabla\left(\b u_t\circ T_t\right)-\nabla\b u_0\right\|_{\b H^1(\Omega)}\leq \frac{\mu}{4K_\Omega}\quad\forall |t|\leq t'$$
	which implies
	$$\left\|\nabla\left(\b u_t\circ T_t\right)\right\|_{\b H^1(\Omega)}\leq \frac{\mu}{4K_\Omega}+\left\|\nabla\b u_0\right\|_{\b L^2(\Omega)}\stackrel{\eqref{e:StatNSSharpParSMallnessU0Recall}}{\leq} \frac{3\mu}{4K_\Omega} \quad\forall |t|\leq t'.$$
	Using as in the proof of Lemma~\ref{l:StatNSSharpParVarExistSol} that $\|\der T_t\|_\infty\leq 1+C|t|$ and $\|\det\der T_t\|_\infty\leq 1+C|t|$ for $|t|\ll1$ we can deduce therefrom the existence of some $c>0$ such that $c<\frac{\mu}{K_\Omega}$ and
		$$\left\|\nabla\b u_t\right\|_{\b H^1(\Omega)}\leq c<\frac{\mu}{K_\Omega}\quad\forall |t|\ll1.$$
		Now the statement follows from Lemma~\ref{l:SharpUniqueStateEquiForSmallU}.
\end{proof}

We thus have proved that local deformations $\varphi_0(t)=\varphi_0\circ T_t^{-1}$ along suitable transformations $T\in{\mathcal T}_{ad}$ of the minimizer $\varphi_0$ still inherit a unique solution of the state equations, thus $\b S_0\left(\varphi_0(t)\right)=\left\{\b u_t\right\}$. Moreover, we know that $t\mapsto\b u_t\circ T_t$ is differentiable at $t=0$ as a mapping into $\b H^1(\Omega)$ and have derived a system that defines the derivative $\partial_t|_{t=0}\left(\b u_t\circ T_t\right)$. And so we can finally formulate first order optimality conditions for the sharp interface problem $\eqref{e:StokesGenMinProblemFctlSharp}-\eqref{e:StatNSSharpConstraintsWeak}$.

Since Corollary~\ref{c:CorollaryUniqueSolSmallDeforStatNS} implies $\b S_0(\varphi_0(t))=\left\{\b u_t\right\}$ for $t$ small enough we can define
$$j_0(\varphi_0(t)):=J_0(\varphi_0(t),\b u_t).$$
Thus we find:

\begin{thm}
	\label{t:StatNSSharpParVarConvOptCondDistProb}
For any minimizer $(\varphi_0,\b u_0)$ of $\eqref{e:StokesGenMinProblemFctlSharp}-\eqref{e:StatNSSharpConstraintsWeak}$ with $\|\nabla\b u_0\|_{\b L^2(\Omega)}\leq\frac{\mu}{2K_\Omega}$ we have the following necessary optimality condition:

\begin{align}\label{e:StatNSSharpParVarConvOptCondLimitDist}\partial_t|_{t=0}j_0\left(\varphi_0\circ T_t^{-1}\right)=-\lambda_0\int_\Omega\varphi_0\div V(0)\dx,\quad \lambda_0\left(\int_\Omega\varphi_0\dx-\beta\left|\Omega\right|\right)=0
\end{align}
for all $T\in{\mathcal T}_{ad}$ with velocity $V\in{\mathcal V}_{ad}$. Here $\lambda_0\geq0$ is a Lagrange multiplier for the integral constraint and the derivative is given by the following formula:
	\begin{equation}\begin{split}\label{e:StatNSSharpVarParResVariationEpsilon}
		&\partial_t|_{t=0}j_0\left(\varphi_0\circ T_t^{-1}\right)=\int_\Omega\left[\der f\left(x,\b u_0,\der\b u_0\right)\left(V(0),\dot{\b u}_0\left[V\right],\der\dot{\b u}_0\left[V\right]-\der\b u_0\der V(0)\right)+\right.\\
		&+\left.f\left(x,\b u_0,\der\b u_0\right)\div V(0)\right]\dx+\gamma c_0\int_\Omega\left(\div V(0)-\nu\cdot\nabla V(0)\nu\right)\,\mathrm d\left|\der\chi_{E_0}\right|
	\end{split}\end{equation}
	with $\nu=\frac{\der\chi_{E_0}}{\left|\der\chi_{E_0}\right|}$ being the generalised unit normal on the Caccioppoli set $E_0=\left\{\varphi_0=1\right\}$, compare \cite{ambrosio}. Moreover $\dot{\b u}_0\left[V\right]\in\b H^1_0(\Omega)$ with $\dot{\b u}_0\left[V\right]=\b 0$ a.e. in $\{\varphi_0=-1\}$ is given as solution of $\eqref{e:StatNSApproxParallelEquUDot}-\eqref{e:StatNSSharpParlVarDivDotBUEpsilonDelta}$.
\end{thm}\begin{proof}
	This follows by using the previous results and direct calculations. The existence of a Lagrange multiplier $\lambda_0\geq0$ for the integral constraint follows as in \cite[Theorem 3]{GarckeHechtStokes}, see also \cite{hecht}.
	
\end{proof}

\begin{remark}\label{r:Hadamard}
	 Assume that $E_0:=\mathrm{int}\left(\left\{\varphi_0=1\right\}\right)$ is a well-defined open subset of $\Omega$ such that $\partial E_0\cap\Omega\in C^2$, $E_0$ has finitely many connected components, $\b g\in\b H^{\frac32}\left(\partial\Omega\right)$ and $(\der_2f\left(\cdot,\b u_0,\der\b u_0\right)-\linebreak[4]\div\der_3f\left(\cdot,\b u_0,\der\b u_0\right))\in\b L^2(E_0)$ for $\b u_0\in\b H^2(E_0)$. Then one can also derive the ``classical'' shape derivatives which can for a large class of possible objective functionals be rewritten in the well-known Hadamard form, compare for instance \cite{bello, delfour,Schmidt_shape_derivative_NavierStokes, sokolowski}. In this case, the optimality conditions derived in Theorem~\ref{t:StatNSSharpParVarConvOptCondDistProb} can be shown to be equivalent to the following system, which can be obtained by classical calculus:
	
	\begin{align*}
	&\int_{E_0}\der\left(f\left(x,\b u_0,\der\b u_0\right)\right)V(0)\dx+\int_\Omega f\left(x,\b u_0,\der\b u_0\right)\div V(0)\dx+\\
	&+\int_{\partial E_0\cap\Omega}\left(\mu\partial_\nu\b q_0\cdot\partial_\nu\b u_0-\left(\der_3f\right)\left(x,\b u_0,\der\b u_0\right)\nu\cdot\partial_\nu\b u_0+\gamma c_0\kappa+2\lambda_0\right) V(0)\cdot\nu\dx=0,
	\end{align*}
which holds for all $V\in\mathcal V_{ad}$. Here, $\b u_0\in\b U^{\varphi_0}$ solves the state equations $\eqref{e:StatNSSharpConstraintsWeak}$ corresponding to $\varphi_0$ and $\b q_0\in\b H^1_0(E_0)$ with $\div\b q_0=0$ is the solution of the adjoint equation
	$$\int_{E_0}\mu\nabla\b q_0\cdot\nabla\b v\dx+b\left(\b v,\b u_0,\b q_0\right)-b\left(\b u_0,\b q_0,\b v\right)=\int_{E_0}\der_{(2,3)}f\left(x,\b u_0,\der\b u_0\right)\left(\b v,\der\b v\right)\dx$$
	which as to hold  for all $\b v\in\b H^1_0(E_0)$ with $\div\b v=0$.	For details, we refer to \cite[Section 26]{hecht}.
\end{remark}

\section{Sharp interface limit, $\epsilon\searrow0$}\label{s:SharpLimit}

In this section, we want to pass to the limit $\epsilon\searrow0$, which means that the interfacial thickness tends to zero and simultaneously also the permeability of the medium outside the fluid region, given by $\left(\alpha_\epsilon\left(-1\right)\right)^{-1}$, tends to zero. When we pass to this limit, we have to consider the state equations, too, and obtain a sequence of velocities depending on the phase field parameter $\epsilon$. Under suitable assumptions, one can show that the sequence converges to a velocity field solving the sharp interface state equation $\eqref{e:StatNSSharpConstraintsWeak}$. To ensure that this limit element coincides with a given velocity field solving $\eqref{e:StatNSSharpConstraintsWeak}$ we need uniqueness of a solution to $\eqref{e:StatNSSharpConstraintsWeak}$ in a minimizer. This is important, since the objective functional may have a different value for two different solutions of $\eqref{e:StatNSSharpConstraintsWeak}$. For a fixed set $E$, one could simply assume smallness of the data and obtain a uniqueness result as for instance in \cite{galdi}. But as we have non-homogeneous boundary data we would have to assume an upper bound on a constant depending on the trace operator on $E$. As we will vary $E$ as a part of the problem, and it is not clear how this constant depends on $E$, this is not the right procedure here. We refer to \cite[Section 11.1, 11.2]{hecht} on details concerning this difficulty.

To overcome this problem, we control the velocity by the objective functional and ensure in this way that $\left\|\b u\right\|_{\b H^1(\Omega)}$ is small enough for the minimizing set $E$, if $\b u$ solves $\eqref{e:StatNSSharpConstraintsWeak}$. Thus, we make the following additional assumption for the remainder of this paper:

\begin{list}{\theAssCount}{\usecounter{AssCount}}\setcounter{AssCount}{\value{AssListCount}
}
 \item\label{a:StatNSUniqunessAss} We assume, that the body force $\b f\in\b L^2(\Omega)$, the boundary term $\b g\in\b H^{\frac12}\left(\partial\Omega\right)$, the viscosity $\mu$ and the objective functional $f$ are chosen such that:
	\begin{enumerate}
		\item there exists some constant $C_u\in\R$ fulfilling
		\begin{align}\label{e:SmallnessCondUniqueConclusion}J_0(\varphi,\b u)\leq C_u\implies \left\|\nabla\b u\right\|_{\b L^2(\Omega)}\leq\frac{\mu}{2K_\Omega}\end{align}
		for all $\varphi\in \Phi_{ad}^0$ and $\b u\in\b S_0(\varphi)$; and
		\item there exists at least one $\varphi_0\in \Phi_{ad}^0$ and $\b u_0\in\b S_0(\varphi_0)$ with
	\begin{align}\label{e:SmallnessCondUniqueCompareFcts}J_0(\varphi_0,\b u_0)\leq C_u.\end{align}
	\end{enumerate}

 \setcounter{AssListCount}{\value{AssCount}}
\end{list}

This requirement will imply unique solvability of the state equations in a neighborhood of the minimizer of $\eqref{e:StokesGenMinProblemFctlSharp}-\eqref{e:StatNSSharpConstraintsWeak}$, see Corollary~\ref{c:StatNSUEpsilonUniqueIfEpsilonSmall}, which will be crucial for the convergence proof, see Theorem~\ref{t:StatNSSharpConvergence}.

\begin{bsp}
	Let's consider the problem of minimizing the total potential power, which leads to the following objective functional in the sharp interface formulation:
	$$J_0\left(\varphi,\b u\right):=\int_\Omega\frac\mu2\left|\nabla\b u\right|^2-\b f\cdot\b u\dx+\gamma c_0P_\Omega\left(E^\varphi\right).$$
	One sees by direct calculations that in this case Assumption~\ref{a:StatNSUniqunessAss} is equivalent to the usual ``smallness of data or high viscosity'' stated in literature concerning uniqueness of the stationary Navier-Stokes equations, cf. \cite{galdi,temam,zeidler4}. Those calculations can for instance be found in \cite[Example 11.1]{hecht}.
	\end{bsp}
	
	We directly see:
	\begin{lem}\label{l:PropMinimizerJ0N}
	Every minimizer $\left(\varphi,\b u\right)$ of the sharp interface problem $\eqref{e:StokesGenMinProblemFctlSharp}-\eqref{e:StatNSSharpConstraintsWeak}$, so in particular $\b u\in\b S_0\left(\varphi\right)$, fulfills
	\begin{align}\label{e:J0NMinimizerSmallnessCond}\left\|\nabla\b u\right\|_{\b L^2(\Omega)}\leq\frac{\mu}{2K_\Omega}.\end{align}
	In particular, this implies by Lemma~\ref{l:SharpUniqueStateEquiForSmallU} that $\b S_0\left(\varphi\right)=\left\{\b u\right\}$.
\end{lem}

\begin{proof}

Assume to have an arbitrary minimizer $\left(\varphi,\b u\right)$ of the sharp interface problem $\eqref{e:StokesGenMinProblemFctlSharp}-\eqref{e:StatNSSharpConstraintsWeak}$. Let $\left(\varphi_c,\b u_c\right)$ be such that
$J_0\left(\varphi_c,\b u_c\right)\leq C_u$ which are given by Assumption~\ref{a:StatNSUniqunessAss}. Then it holds, since $\left(\varphi,\b u\right)$ minimize $J_0$ in particular
$J_0(\varphi,\b u)\leq J_0\left(\varphi_c,\b u_c\right)\leq C_u$ and so by $\eqref{e:SmallnessCondUniqueConclusion}$ we deduce 
$\left\|\nabla\b u\right\|_{\b L^2(\Omega)}\leq\frac{\mu}{2K_\Omega}$ which proves $\eqref{e:J0NMinimizerSmallnessCond}$.
\end{proof}

\begin{remark}
	Using the results of Lemma~\ref{l:PropMinimizerJ0N}, we see that for a minimizer $\left(\varphi,\b u\right)\in L^1(\Omega)\times\b H^1(\Omega)$ of the sharp interface problem, the state equations $\eqref{e:StatNSSharpConstraintsWeak}$ corresponding to $\varphi$ have due to Lemma~\ref{l:SharpUniqueStateEquiForSmallU} always a \emph{unique} solution, thus $\b S_0(\varphi)=\left\{\b u\right\}$.\\
	This will play an essential role when showing that minimizers of $\left(J_\epsilon\right)_{\epsilon>0}$ converge to a minimizer of $J_0$, see Theorem~\ref{t:StatNSSharpConvergence}.
\end{remark}

	Additionally, we need for the sharp interface convergence the radially unboundedness of the objective functional with respect to the velocity. Hence the following assumption is necessary for the remainder of this work:
	
\begin{list}{\theAssCount}{\usecounter{AssCount}}\setcounter{AssCount}{\value{AssListCount}
}
 \item\label{a:RadiallyUnbounded} We assume, that $F:\b U\to\R$, $F(\b u):=\int_\Omega f\left(x,\b u(x),\der\b u(x)\right)\dx$ is radially unbounded, i.e. for any sequence $(\b u_k)_{k\in\N}\subset\b U$ with $\lim_{k\to\infty}\|\b u_k\|_{\b H^1(\Omega)}=\infty$ we have $\lim_{k\to\infty}F(\b u_k)=+\infty$.

 \setcounter{AssListCount}{\value{AssCount}}
\end{list}

	\subsection{Convergence of minimizers}\label{s:ConvMinim}
	The first main result concerning the sharp interface limit is given by the following theorem:

\begin{thm}\label{t:StatNSSharpConvergence}
Let $\left(\varphi_\epsilon,\b u_\epsilon\right)_{\epsilon>0}\subseteq L^1(\Omega)\times\b H^1(\Omega)$ be minimizers of the phase field problems $\eqref{e:StatNSGenMinProblemFctl}-\eqref{e:StatNSGenConstraintsWeak}$. Then there exists a subsequence, denoted by the same, and an element $\left(\varphi_0,\b u_0\right)\in L^1(\Omega)\times\b H^1(\Omega)$ such that
\begin{align*}
\varphi_\epsilon\stackrel{\epsilon\searrow0}{\rightarrow}\varphi_0\quad\text{in }L^1(\Omega),\qquad \b u_\epsilon\stackrel{\epsilon\searrow0}{\rightharpoonup}\b u_0\quad\text{in }\b H^1(\Omega).
\end{align*}

If it holds	\begin{align}\label{e:StatNsGammaLimitGrothcondOnMinim}\left\|\varphi_\epsilon-\varphi_0\right\|_{L^1\left(\left\{x\in\Omega\mid \varphi_0(x)=1,\varphi_\epsilon(x)<0\right\}\right)}=\mathcal O\left(\epsilon\right)
	\end{align}
then we obtain additionally $\lim_{\epsilon\searrow0}\|\b u_\epsilon-\b u_0\|_{\b H^1(\Omega)}=0$. Moreover, $\left(\varphi_0,\b u_0\right)$ is then a minimizer of the sharp interface problem $\eqref{e:StokesGenMinProblemFctlSharp}-\eqref{e:StatNSSharpConstraintsWeak}$ and 
\begin{align}\label{e:StatNSSharpConvObjFctl}\lim_{\epsilon\searrow0}J_\epsilon\left(\varphi_\epsilon,\b u_\epsilon\right)=J_0\left(\varphi_0,\b u_0\right).\end{align}
\end{thm}
\begin{remark}
The existence of minimizers $\left(\varphi_\epsilon,\b u_\epsilon\right)$ for the phase field problems $\eqref{e:StatNSGenMinProblemFctl}-\eqref{e:StatNSGenConstraintsWeak}$ for every $\epsilon>0$ follows by Theorem~\ref{l:StatNSDiffuseGenerelFctExistMini}. Thus, using the statement of Theorem~\ref{t:StatNSSharpConvergence}, it follows in particular the existence of a minimizer for the sharp interface problem $\eqref{e:StokesGenMinProblemFctlSharp}-\eqref{e:StatNSSharpConstraintsWeak}$ if $\eqref{e:StatNsGammaLimitGrothcondOnMinim}$ is fulfilled for a sequence of minimizers. This has not been shown so far and is still an open problem for the general shape optimization problem in fluid dynamics, compare also discussion in the introduction and in Remark \ref{r:ExistMinj0}. And so proving a convergence result without any condition as in $\eqref{e:StatNsGammaLimitGrothcondOnMinim}$ would imply a much stronger result concerning well-posedness of the shape optimization problem that is not expected. In this sense, the result at hand seems currently optimal.
\end{remark}

Before proving this theorem we start with two preparatory lemmas.

\begin{lem}\label{l:FBEpsilonSharpConv}
 Let $\left(\varphi_\epsilon\right)_{\epsilon>0}\subseteq L^1(\Omega)$, $\left|\varphi_\epsilon\right|\leq1$ a.e., with $\b u_\epsilon\in\b S_\epsilon\left(\varphi_\epsilon\right)$ for all $\epsilon>0$ be given such that $\lim_{\epsilon\searrow0}\left\|\varphi_\epsilon-\varphi_0\right\|_{L^1(\Omega)}=0$ together with the convergence rate $\eqref{e:StatNsGammaLimitGrothcondOnMinim}$ where $\varphi_0\in BV\left(\Omega,\left\{\pm1\right\}\right)$ and $\b U^{\varphi_0}\neq\emptyset$. Assume moreover $\sup_{\epsilon>0}\|\b u_\epsilon\|_{\b H^1(\Omega)}<\infty$. Then there exists a subsequence of $\left(\varphi_\epsilon,\b u_\epsilon\right)_{\epsilon>0}$ (denoted by the same) and some $\b u_0\in\b S_0\left(\varphi_0\right)$ such that \begin{align}\label{e:StatNSSharpModelHelpLemmaConvRes}\lim_{\epsilon\searrow0}\left\|\b u_\epsilon-\b u_0\right\|_{\b H^1(\Omega)}=0,\quad\lim_{\epsilon\searrow0}\int_\Omega\alpha_\epsilon\left(\varphi_\epsilon\right)\left|\b u_\epsilon\right|^2\dx=0.\end{align}
  
\end{lem}
\begin{proof}
	We skip some details which can be found in \cite[Lemma 3]{GarckeHechtStokes} and mainly point out the differences that occur when dealing with the nonlinearity in the state equation.\\
	
	We start by choosing a subsequence of $\left(\varphi_\epsilon\right)_{\epsilon>0}$ that converges pointwise almost everywhere to $\varphi_0$ in $\Omega$. Then we get as in \cite{GarckeHechtStokes} that it holds $\lim_{\epsilon\searrow0}\alpha_\epsilon\left(\varphi_\epsilon\left(x\right)\right)=\alpha_0\left(\varphi_0\left(x\right)\right)$ for a.e. $x\in\Omega$. Moreover, we see as in \cite[Lemma 2]{GarckeHechtStokes} that we can deduce 
	\begin{align}\label{e:TermGeosToZero}\lim_{\epsilon\searrow0}\int_\Omega\alpha_\epsilon\left(\varphi_\epsilon\right)\left|\b v\right|^2\dx=0\quad\forall\b v\in\b H^1(\Omega),\b v|_{\Omega\setminus E^{\varphi_0}}=\b 0\end{align}
	from the convergence rate given by $\eqref{e:StatNsGammaLimitGrothcondOnMinim}$ and the convergence rate on $\alpha_\epsilon$ given by Assumption~\ref{a:Alpha}.\\
Next we notice that $\b u_\epsilon\in\b U$ are for all $\epsilon>0$ the unique solutions of
	$$\min_{\b v\in\b U}FP_\epsilon\left(\b v\right):=\int_\Omega\left(\frac12\alpha_\epsilon\left(\varphi_\epsilon\right)\left|\b v\right|^2+\frac\mu2\left|\nabla\b v\right|^2+\b u_\epsilon\cdot\nabla\b u_\epsilon\cdot\b v-\b f\cdot\b v\right)\dx$$
	since the state equations $\eqref{e:StatNSGenConstraintsWeak}$ are the necessary and sufficient first order optimality conditions for these optimization problems. \\
	From the boundedness of $\left(\b u_\epsilon\right)_{\epsilon>0}$ in $\b H^1(\Omega)$ we can find a subsequence that converges weakly in $\b H^1(\Omega)$ and pointwise almost everywhere to some limit element $\b u_0\in\b U$ as $\epsilon\searrow0$.\\	
We then define $FP_0:\b H^1(\Omega)\to\overline\R$ by
	$$FP_0\left(\b v\right):=\int_\Omega\left(\frac12\alpha_0\left(\varphi_0\right)\left|\b v\right|^2+\b u_0\cdot\nabla\b u_0\cdot\b v+\frac\mu2\left|\nabla\b v\right|^2-\b f\cdot\b v\right)\dx$$
	
	\noindent and claim, that $\left(FP_\epsilon\right)_{\epsilon>0}$ $\Gamma$-converges to $FP_0$ as $\epsilon\searrow0$ in $\b U$ equipped with the weak $\b H^1(\Omega)$ topology. We notice particularly that $FP_0\not\equiv\infty$ as $\b U^{\varphi_0}\neq\emptyset$.\\	
	Using the continuity properties of the trilinear form $b$ (compare Lemma~\ref{l:TrilinearFormStrongCont}) we get from $\eqref{e:TermGeosToZero}$ with similar arguments as in \cite{GarckeHechtStokes} that for any $\b v\in\b U$ it holds
	$\limsup_{\epsilon\searrow0}FP_\epsilon\left(\b v\right)\leq FP_0\left(\b v\right)$. Thus the constant sequence defines a recovery sequence.\\
	For showing the lower semicontinuity condition, let $\left(\b v_\epsilon\right)_{\epsilon>0}\subseteq\b U$ be an arbitrary sequence that converges weakly in $\b H^1(\Omega)$ to some $\b v\in\b U$. By using similar ideas as in Lemma \ref{l:TrilinearFormStrongCont} we can establish $
		\lim_{\epsilon\searrow0}\left|b\left(\b u_\epsilon,\b u_\epsilon,\b v_\epsilon\right)-b\left(\b u_0,\b u_0,\b v\right)\right|=0$.	The remaining terms can be considered as in \cite{GarckeHechtStokes} and we obtain
 	$FP_0(\b v)\leq\liminf_{\epsilon\searrow0}FP_\epsilon\left(\b v_\epsilon\right)$
 	which proves that $\left(FP_\epsilon\right)_{\epsilon>0}$ $\Gamma$-converges to $FP_0$ as $\epsilon\searrow0$ in $\b U$ equipped with the weak $\b H^1(\Omega)$-topology.\\
 	Applying standard results on $\Gamma$-convergence, see for instance \cite{dalmaso}, we can conclude that $\left(\b u_\epsilon\right)_{\epsilon>0}$ is converging weakly in $\b H^1(\Omega)$ to the unique minimizer of $FP_0$, which implies that $\b u_0$ minimizes $FP_0$. But, considering the necessary and sufficient first order optimality conditions for this convex optimization problem, this implies that $\b u_0$ fulfills the state equation $\eqref{e:StatNSSharpConstraintsWeak}$ and this implies $\b u_0\in\b S_0\left(\varphi_0\right)$.\\
 	Besides, the $\Gamma$-convergence result gives then $\lim_{\epsilon\searrow0}FP_\epsilon\left(\b u_\epsilon\right)=FP_0(\b u_0)$. As one can show $\lim_{\epsilon\searrow0}b\left(\b u_\epsilon,\b u_\epsilon,\b u_\epsilon\right)=b\left(\b u_0,\b u_0,\b u_0\right)$ we get therefrom the convergences $\eqref{e:StatNSSharpModelHelpLemmaConvRes}$. This proves the lemma.
 	
\end{proof}

We state another variant of this lemma, where the uniform bound on the velocities is not part of the assumption, but instead we have more information about the limit element of the phase field variables:

\begin{lem}\label{l:FBEpsilonSharpConvWithoutBound}
 Let $\left(\varphi_\epsilon\right)_{\epsilon>0}\subseteq L^1(\Omega)$ and $\varphi_0\in BV(\Omega,\{\pm1\})$ be as in Lemma~\ref{l:FBEpsilonSharpConv} and $\b u_\epsilon\in\b S_\epsilon(\varphi_\epsilon)$. But instead of the uniform bound on $(\b u_\epsilon)_{\epsilon>0}$ assume that there exists some $\b u\in\b U^{\varphi_0}$ and a constant $0<\overline c<\mu$ fulfilling
   \begin{align}\label{e:StatNSEpsConvCondOnLimitElement}\left|\int_\Omega\b v\cdot\nabla\b u\cdot\b v\dx\right|\leq \overline c\left\|\nabla\b v\right\|_{\b L^2(\Omega)}^2\quad\forall\b v\in\b H^1_0(\Omega).\end{align}
 
  Then there exists a subsequence of $\left(\varphi_\epsilon,\b u_\epsilon\right)_{\epsilon>0}$ (denoted by the same) and some $\b u_0\in\b S_0\left(\varphi_0\right)$, such that $\eqref{e:StatNSSharpModelHelpLemmaConvRes}$ is fulfilled.
\end{lem}
\begin{proof}
	We want to apply Lemma~\ref{l:FBEpsilonSharpConv} and thus have to show that there exists a uniform bound on $\left\|\b u_\epsilon\right\|_{\b H^1(\Omega)}$. To do this, let $\b u\in\b U^{\varphi_0}$ be chosen such that $\eqref{e:StatNSEpsConvCondOnLimitElement}$ is fulfilled. We obtain from the state equations $\eqref{e:StatNSGenConstraintsWeak}$, written for $\b u_\epsilon\in\b S_\epsilon\left(\varphi_\epsilon\right)$, that for $\b w_\epsilon:=\b u_\epsilon-\b u\in\b V$ it holds
	
	\begin{align*}
		\int_\Omega\alpha_\epsilon\left(\varphi_\epsilon\right)\b w_\epsilon\cdot\b v+\mu\nabla\b w_\epsilon\cdot\nabla\b v\dx+b\left(\b w_\epsilon,\b w_\epsilon,\b v\right)+b\left(\b w_\epsilon,\b u,\b v\right)+b\left(\b u,\b w_\epsilon,\b v\right)=\\
		=\int_\Omega\b f\cdot\b v-\alpha_\epsilon\left(\varphi_\epsilon\right)\b u\cdot\b v-\mu\nabla\b u\cdot\nabla\b v\dx-b\left(\b u,\b u,\b v\right)\quad\forall\b v\in\b V.
	\end{align*}
	
	We can insert $\b w_\epsilon\in\b V$ as a test function into this equation and obtain with similar calculations as in \cite[Theorem IX.4.1]{galdi}
	
	\begin{align*}
	&\int_\Omega\alpha_\epsilon\left(\varphi_\epsilon\right)\left|\b w_\epsilon\right|^2\dx+\mu\left\|\nabla\b w_\epsilon\right\|_{\b L^2(\Omega)}^2+b\left(\b w_\epsilon,\b u,\b w_\epsilon\right)=\\
	&=\int_\Omega\b f\cdot\b w_\epsilon-\alpha_\epsilon\left(\varphi_\epsilon\right)\b u\cdot\b w_\epsilon-\mu\nabla\b u\cdot\nabla\b w_\epsilon\dx+b\left(\b u,\b u,\b w_\epsilon\right).
	\end{align*}
	Applying the inequalities of Young, H\"older and Poincar\'e this gives with $\eqref{e:StatNSEpsConvCondOnLimitElement}$
	
	\begin{align*}
	&\int_\Omega\alpha_\epsilon\left(\varphi_\epsilon\right)\left|\b w_\epsilon\right|^2\dx+\mu\left\|\nabla\b w_\epsilon\right\|_{\b L^2(\Omega)}^2\leq\left\|\b f\right\|_{\b L^2(\Omega)}\left\|\b w_\epsilon\right\|_{\b L^2(\Omega)}+ \frac12\int_\Omega\alpha_\epsilon\left(\varphi_\epsilon\right)\left|\b w_\epsilon\right|^2\dx+\\
&+\left(\limsup_{\epsilon\searrow0}\frac12\int_\Omega\alpha_\epsilon\left(\varphi_\epsilon\right)\left|\b u\right|^2\dx+c\right)+\mu\left\|\nabla\b u\right\|_{\b L^2(\Omega)}\left\|\nabla\b w_\epsilon\right\|_{\b L^2(\Omega)}+\\
&+C\left\|\b u\right\|_{\b H^1(\Omega)}^2\left\|\nabla\b w_\epsilon\right\|_{\b L^2(\Omega)}+\overline c\left\|\nabla\b w_\epsilon\right\|_{\b L^2(\Omega)}^2
	\end{align*}
	which holds for $C, c\geq 0$ independent of $\epsilon$ and $\epsilon>0$ small enough.\\
	Thus we get, after applying Young's inequality, a constant $C>0$ independent of $\epsilon>0$, such that
\begin{equation}\begin{split}\label{e:StatNSSharpModelEstimateIrgendwasMitIrgendwasFPAbschaetzungen2}
&\int_\Omega\alpha_\epsilon\left(\varphi_\epsilon\right)\left|\b u_\epsilon\right|^2+\left\|\nabla\b u_\epsilon\right\|_{\b L^2(\Omega)}^2\leq C\left(\limsup_{\epsilon\searrow0}\frac12\int_\Omega\alpha_\epsilon\left(\varphi_\epsilon\right)\left|\b u\right|^2\dx+1\right)
\end{split}\end{equation}
for all $\epsilon>0$ small enough. Using the considerations of \cite[Lemma 3]{GarckeHechtStokes} we find \linebreak[4]$\limsup_{\epsilon\searrow0}\int_\Omega\alpha_\epsilon\left(\varphi_\epsilon\right)\left|\b u\right|^2\dx=0$ and so we can deduce from $\eqref{e:StatNSSharpModelEstimateIrgendwasMitIrgendwasFPAbschaetzungen2}$ and Poincar\'e's inequality that there exists a constant $C>0$ independent of $\epsilon$ such that $\left\|\b u_\epsilon\right\|_{\b H^1(\Omega)}<C.$ We can now complete the proof by applying Lemma~\ref{l:FBEpsilonSharpConv}.
	
\end{proof}

Finally, we can show Theorem~\ref{t:StatNSSharpConvergence}:

\begin{proof}[Proof of Theorem~\ref{t:StatNSSharpConvergence}]
We split the proof into several steps and use the ideas of the proof of \cite[Theorem 2]{GarckeHechtStokes}.

\begin{itemize}

\item \textit{1st step:} Assume that $\left(\varphi,\b u\right)\in L^1(\Omega)\times\b H^1(\Omega)$ is an arbitrary pair such that
$J_0\left(\varphi,\b u\right)\leq C_u$ and thus, due to $\eqref{e:SmallnessCondUniqueConclusion}$, in particular
\begin{align}\label{e:J0NMinStatNSProofSmallnesCons}\left\|\nabla\b u\right\|_{\b L^2(\Omega)}\leq\frac{\mu}{2K_\Omega}.\end{align} We follow the construction of \cite[1st step, Proof of Theorem 2]{GarckeHechtStokes} to obtain a sequence $\left(\varphi_\epsilon\right)_{\epsilon>0}\subset\Phi_{ad}$ such that 
\begin{align}\label{e:StokesGammaConvLimSupEpsilonK}\limsup_{\epsilon\searrow0}\int_\Omega\left(\frac{\gamma\epsilon}{2}\left|\nabla\varphi_\epsilon\right|^2+\frac\gamma\epsilon\psi\left(\varphi_\epsilon\right)\right)\dx\leq\gamma c_0P_\Omega\left(\{\varphi=1\}\right)\end{align}
analog as it is done for example in \cite[p. 222 ff]{sternberg}, \cite[Proposition 2]{modica} or \cite[Proposition 3.11]{blowey_elliot}. From this we obtain in particular $\|\varphi_\epsilon-\varphi\|_{L^1(\Omega)}=\mathcal O(\epsilon)$.	Then we choose some $\b u_\epsilon\in\b S_{\epsilon}\left(\varphi_{\epsilon}\right)$. By using $\eqref{e:ContinuityEstimateTrilinearForm}$, we observe that $\eqref{e:J0NMinStatNSProofSmallnesCons}$ implies $\eqref{e:StatNSEpsConvCondOnLimitElement}$ and so we can apply Lemma~\ref{l:FBEpsilonSharpConvWithoutBound} to find that, after possible choosing a subsequence, $\left(\b u_\epsilon\right)_{\epsilon>0}$ converges strongly in $\b H^1(\Omega)$ to some $\b u_0\in\b S_0(\varphi)=\left\{\b u\right\}$, thus $\b u_0\equiv\b u$, and it holds
$\lim_{\epsilon\searrow0}\int_\Omega\alpha_{\epsilon}\left(\varphi_{\epsilon}\right)\left|\b u_\epsilon\right|^2\dx=0.$ Using the continuity of the objective functional we end up with
$$\limsup_{\epsilon\searrow0}J_{\epsilon}(\varphi_{\epsilon},\b u_\epsilon)\leq J_0(\varphi,\b u).$$

\item\textit{2nd step:} Next we will show that for any sequence $\left(\varphi_\epsilon,\b u_\epsilon\right)_{\epsilon>0}\subseteq L^1(\Omega)\times\b H^1(\Omega)$ such that $\left(\varphi_\epsilon\right)_{\epsilon>0}$ converges strongly in $L^1(\Omega)$ to some $\varphi\in L^1(\Omega)$ fulfilling

\begin{align}\label{e:StatNsShaprConvL1InProofconvRate}\left\|\varphi_\epsilon-\varphi\right\|_{L^1\left(\left\{x\in\Omega\mid \varphi_0(x)=1,\varphi_\epsilon(x)<0\right\}\right)}=\mathcal O\left(\epsilon\right)
	\end{align}
	
and $\left(\b u_\epsilon\right)_{\epsilon>0}$ converges weakly in $\b H^1(\Omega)$ to some $\b u\in\b H^1(\Omega)$ it holds
$$J_0(\varphi,\b u)\leq\liminf_{\epsilon\searrow0}J_\epsilon\left(\varphi_\epsilon,\b u_\epsilon\right).$$

Without loss of generality we assume $\liminf_{\epsilon\searrow0} J_\epsilon(\varphi_\epsilon,\b u_\epsilon)<\infty$ and $\varphi\in BV\left(\Omega,\left\{\pm1\right\}\right)$ with $\int_\Omega\varphi\dx\leq\beta|\Omega|$.\\
We can assume that (after choosing a subsequence) $\left(\varphi_\epsilon\right)_{\epsilon>0}$ and $\left(\b u_\epsilon\right)_{\epsilon>0}$ converge pointwise almost everywhere in $\Omega$, and thus using Fatou's Lemma, we see
$$\int_\Omega\alpha_0\left(\varphi\right)\left|\b u\right|^2\dx\leq\liminf_{\epsilon\searrow0}\int_\Omega\alpha_\epsilon\left(\varphi_\epsilon\right)\left|\b u_\epsilon\right|^2\dx<\infty$$
and so in particular $\b u=\b0$ a.e. in $\Omega\setminus E^\varphi$. Thanks to $\b u_\epsilon\in\b S_\epsilon\left(\varphi_\epsilon\right)$ we see $\b u_\epsilon\in\b U$ for all $\epsilon>0$ and deduce $\b u\in\b U$. Altogether this implies $\b u\in\b U^\varphi$ and thus $\b U^\varphi\neq\emptyset$.\\
	Using \cite[Proposition 1]{modica} we get after rescaling in $\epsilon$ that
			$$\gamma c_0P_\Omega\left(\{\varphi=1\}\right)\leq\liminf_{\epsilon\searrow0}\int_\Omega\left(\frac{\gamma\epsilon}{2}\left|\nabla\varphi_\epsilon\right|^2+\frac\gamma\epsilon\psi\left(\varphi_\epsilon\right)\right)\dx.$$
			We choose then a subsequence $\left(J_{\epsilon_k}\left(\varphi_{\epsilon_k},\b u_{\epsilon_k}\right)\right)_{k\in\N}$ such that
			$\lim_{k\to\infty}J_{\epsilon_k}\left(\varphi_{\epsilon_k},\b u_{\epsilon_k}\right)=\liminf_{\epsilon\searrow0}J_\epsilon\left(\varphi_\epsilon,\b u_\epsilon\right).$			With the help of the convergence rate on $(\varphi_\epsilon)_{\epsilon>0}$ and using 
			$\sup_{k\in\N}\left\|\b u_{\epsilon_k}\right\|_{\b H^1(\Omega)}<\infty$, which follows from the weak convergence of $\left(\b u_\epsilon\right)_{\epsilon>0}$ in $\b H^1(\Omega)$,	we thus can apply Lemma~\ref{l:FBEpsilonSharpConv} and get a subsequence $\left(J_{\epsilon_{k(l)}}\left(\varphi_{\epsilon_{k(l)}},\b u_{\epsilon_{k(l)}}\right)\right)_{l\in\N}$ such that 
		$$\lim_{l\rightarrow\infty}\left\|\b u_{\epsilon_{k(l)}}-\b u\right\|_{\b H^1\left(\Omega\right)}=0,\quad \lim_{l\rightarrow\infty}\int_\Omega\alpha_{\epsilon_{k(l)}}\left(\varphi_{\epsilon_{k(l)}}\right)\left|\b u_{\epsilon_{k(l)}}\right|^2\dx=0.$$
		Plugging these results together we end up with		
		\begin{align*}
	J_0\left(\varphi,\b u\right)&=\int_\Omega f\left(x,\b u,\der\b u\right)\dx+\gamma c_0P_\Omega\left(\{\varphi=1\}\right)\leq\liminf_{l\rightarrow\infty} J_{\epsilon_{k(l)}}\left(\varphi_{\epsilon_{k(l)}},\b u_{\epsilon_{k(l)}}\right)=\\
			&=\lim_{k\rightarrow\infty}J_{\epsilon_k}\left(\varphi_{\epsilon_k},\b u_{\epsilon_k}\right)=\liminf_{\epsilon\searrow0}J_\epsilon\left(\varphi_\epsilon,\b u_\epsilon\right)\end{align*}
		and finish the second step.\\
		
		\item\textit{3rd step:} Now let $\left(\varphi_\epsilon,\b u_\epsilon\right)_{\epsilon>0}\subseteq L^1(\Omega)\times\b H^1(\Omega)$ be minimizers of $\eqref{e:StatNSGenMinProblemFctl}-\eqref{e:StatNSGenConstraintsWeak}$. By Assumption~\ref{a:StatNSUniqunessAss} we know that there exists some $\left(\widetilde\varphi,\widetilde{\b u}\right)\in \Phi_{ad}^0\times\b U$ with $\widetilde{\b u}\in\b S_0\left(\widetilde\varphi\right)$ and
		\begin{align}\label{e:StatNSSharpConvSmallConvFctJ0N}J_0\left(\widetilde\varphi,\widetilde{\b u}\right)\leq C_u.\end{align}
		This gives in view of $\eqref{e:SmallnessCondUniqueConclusion}$ in particular
		$\left\|\nabla\widetilde{\b u}\right\|_{\b L^2(\Omega)}\leq\frac{\mu}{2K_\Omega}$	and thus by Lemma~\ref{e:ContinuityEstimateTrilinearForm} also $b\left(\b v,\widetilde{\b u},\b v\right)\leq \frac\mu2\left\|\nabla\b v\right\|_{\b L^2(\Omega)}^2$ for all $\b v\in\b H^1_0(\Omega).$ From $\eqref{e:StatNSSharpConvSmallConvFctJ0N}$ we find that we can apply the third part of this proof and obtain a sequence $\left(\widetilde\varphi_\epsilon,\widetilde{\b u}_\epsilon\right)_{\epsilon>0}\subset L^1(\Omega)\times\b H^1(\Omega)$ converging in $L^1(\Omega)\times\b H^1(\Omega)$ to $\left(\widetilde\varphi,\widetilde{\b u}\right)$ such that
		
		
		\begin{align}\label{e:StatNSSHarpConvHelpFct1}\limsup_{\epsilon\searrow0}J_\epsilon\left(\widetilde\varphi_\epsilon,\widetilde{\b u}_\epsilon\right)\leq J_0\left(\widetilde\varphi,\widetilde{\b u}\right)\leq C_u\end{align}
		and in particular $\sup_{\epsilon>0} J_\epsilon\left(\widetilde\varphi_\epsilon,\widetilde{\b u}_\epsilon\right)<\infty.$	From the fact that $\left(\varphi_\epsilon, \b u_\epsilon\right)$ minimize $J_\epsilon$ for every $\epsilon>0$ we know that 
	\begin{align}\label{e:StatNSSharpConvJNEPsilonMinBoundedUniform}
	&J_\epsilon\left(\varphi_\epsilon,\b u_\epsilon\right)\leq J_\epsilon\left(\widetilde\varphi_\epsilon,\widetilde{\b u}_\epsilon\right)<C\end{align}
	where $C>0$ is a constant independent of $\epsilon>0$. Therefrom
		\begin{align}\label{e:StatNSConvResultNormGLEpsilonBound}\sup_{\epsilon>0}\int_\Omega\left(\frac{\gamma\epsilon}{2}\left|\nabla\varphi_\epsilon\right|^2+\frac\gamma\epsilon\psi\left(\varphi_\epsilon\right)\right)\dx<\infty\end{align}
	and by Assumption~\ref{a:RadiallyUnbounded} also
	\begin{align}\label{e:StatNSConvResultNormUEpsilonBound}\sup_{\epsilon>0}\left\|\b u_\epsilon\right\|_{\b H^1(\Omega)}<\infty.\end{align}
	
	Now using the arguments of \cite[Proposition 3, case a)]{modica} we get from  $\eqref{e:StatNSConvResultNormGLEpsilonBound}$ that $\left(\varphi_\epsilon\right)_{\epsilon>0}$ has a subsequence, denoted by the same, that converges in $L^1(\Omega)$ to an element $\varphi_0\in L^1(\Omega)$. Besides, we find that $\left(\b u_\epsilon\right)_{\epsilon>0}$ has a subsequence that converges weakly in $\b H^1(\Omega)$ to some $\b u_0\in\b U$.\\
	If we assume, that the sequence of minimizers fulfills the convergence rate $\eqref{e:StatNsGammaLimitGrothcondOnMinim}$ we see from the second step of this proof, that it holds
	\begin{align}\label{e:StatNSSHarpConvHelpFct2}J_0\left(\varphi_0,\b u_0\right)\leq\liminf_{\epsilon\searrow0} J_\epsilon\left(\varphi_\epsilon,\b u_\epsilon\right).\end{align}
	
	We want to show, that $\left(\varphi_0,\b u_0\right)$ are a minimizer for $\eqref{e:StokesGenMinProblemFctlSharp}-\eqref{e:StatNSSharpConstraintsWeak}$. For this purpose, let $\left(\varphi,\b u\right)$ be another arbitrary pair. To show that
	$J_0\left(\varphi_0,\b u_0\right)\leq J_0\left(\varphi,\b u\right)$
	we can assume without loss of generality that $J_0\left(\varphi,\b u\right)\leq C_u$, since by $\eqref{e:StatNSSHarpConvHelpFct1}$, $\eqref{e:StatNSSharpConvJNEPsilonMinBoundedUniform}$ and $\eqref{e:StatNSSHarpConvHelpFct2}$ we have
	\begin{align}\label{e:StatNSSHarpConvJ0nCuEstimate}J_0\left(\varphi_0,\b u_0\right)\leq C_u.\end{align}
Consequently, the first step of this proof guarantees the existence of a sequence $\left(\overline\varphi_\epsilon,\overline{\b u}_\epsilon\right)_{\epsilon>0}\subseteq L^1(\Omega)\times\b H^1(\Omega)$ converging to $\left(\varphi,\b u\right)$ in $L^1(\Omega)\times\b H^1(\Omega)$ such that
	$\limsup_{\epsilon\searrow0}J_\epsilon\left(\overline\varphi_\epsilon,\overline{\b u}_\epsilon\right)\leq J_0\left(\varphi,\b u\right).$	Combining those result, we obtain
	
	\begin{align}\label{e:StatNSSharpConvJ0NNochEinEstimateInProof}J_0\left(\varphi_0,\b u_0\right)\leq\liminf_{\epsilon\searrow0} J_\epsilon\left(\varphi_\epsilon,\b u_\epsilon\right)\leq\limsup_{\epsilon\searrow0}J_\epsilon\left(\overline\varphi_\epsilon,\overline{\b u}_\epsilon\right)\leq J_0\left(\varphi,\b u\right)\end{align}
	the second inequality being a consequence of $\left(\varphi_\epsilon,\b u_\epsilon\right)$ minimizing $J_\epsilon$ for every $\epsilon>0$.

	As $\left(\varphi,\b u\right)$ has been arbitrary this implies that $\left(\varphi_0,\b u_0\right)$ is a minimizer of $J_0$.\\
	To deduce the statement of the theorem, it remains to show the strong convergence of $\left(\b u_\epsilon\right)_{\epsilon>0}$ in $\b H^1(\Omega)$ and $\eqref{e:StatNSSharpConvObjFctl}$. For this purpose, we use again $\eqref{e:StatNsGammaLimitGrothcondOnMinim}$ and consequently can apply Lemma~\ref{l:FBEpsilonSharpConv} to deduce that $\left(\b u_\epsilon\right)_{\epsilon>0}$ converges strongly in $\b H^1(\Omega)$ and 
	\begin{align}\label{e:StatNSSHaprConvPenTermZero}\lim_{\epsilon\searrow0}\int_\Omega\alpha_\epsilon\left(\varphi_\epsilon\right)\left|\b u_\epsilon\right|^2\dx=0.\end{align}
	By the first step of this proof and $\eqref{e:StatNSSHarpConvJ0nCuEstimate}$ we find a sequence $\left(\widehat\varphi_\epsilon,\widehat{\b u}_\epsilon\right)_{\epsilon>0}\subset L^1(\Omega)\times\b H^1(\Omega)$ converging to $\left(\varphi_0,\b u_0\right)$ strongly in $L^1(\Omega)\times\b H^1(\Omega)$ such that
$\limsup_{\epsilon\searrow0}J_\epsilon\left(\widehat\varphi_\epsilon,\widehat{\b u}_\epsilon\right)\leq J_0\left(\varphi_0,\b u_0\right).$	Then we see similar to $\eqref{e:StatNSSharpConvJ0NNochEinEstimateInProof}$ by applying $\eqref{e:StatNSSHarpConvHelpFct2}$ that
	\begin{align*}J_0\left(\varphi_0,\b u_0\right)&\leq\liminf_{\epsilon\searrow0}J_\epsilon\left(\varphi_\epsilon,\b u_\epsilon\right)\leq\limsup_{\epsilon\searrow0}J_\epsilon\left(\widehat\varphi_\epsilon,\widehat{\b u}_\epsilon\right)\leq J_0\left(\varphi_0,\b u_0\right)\end{align*}
	and can finally deduce $\eqref{e:StatNSSharpConvObjFctl}$.
\end{itemize}

\end{proof}

Using this result, we can now show that for a minimizer $\left(\varphi_\epsilon,\b u_\epsilon\right)$ of $\eqref{e:StatNSGenMinProblemFctl}-\eqref{e:StatNSGenConstraintsWeak}$ the state equations corresponding to $\varphi_\epsilon$ have a \emph{unique} solution if $\epsilon>0$ is small enough and $\eqref{e:StatNsGammaLimitGrothcondOnMinim}$ is fulfilled, as the following corollary shows:

\begin{cor}\label{c:StatNSUEpsilonUniqueIfEpsilonSmall}
	Assume $\left(\varphi_\epsilon,\b u_\epsilon\right)\in L^1(\Omega)\times\b H^1(\Omega)$ are minimizer of the phase field problems $\eqref{e:StatNSGenMinProblemFctl}-\eqref{e:StatNSGenConstraintsWeak}$ such that $\eqref{e:StatNsGammaLimitGrothcondOnMinim}$ is fulfilled. Then, for $\epsilon>0$ small enough, it holds
	$\b S_\epsilon(\varphi_\epsilon)=\left\{\b u_\epsilon\right\}.$	This means, that the solution of $\eqref{e:StatNSSharpConstraintsWeak}$ corresponding to $\varphi_\epsilon$ is \emph{unique}. Moreover, we have
	$\left\|\nabla\b u_\epsilon\right\|_{\b L^2(\Omega)}<\frac{\mu}{K_\Omega}.$
\end{cor}
\begin{proof}
	It follows from Theorem~\ref{t:StatNSSharpConvergence}, that
	$\left\|\b u_\epsilon-\b u\right\|_{\b H^1(\Omega)}<\delta$
	for some $0<\delta<\frac{\mu}{2K_\Omega}$, if $\epsilon>0$ is small enough, where $\b u\in\b S_0(\varphi)$ and $\left(\varphi,\b u\right)$ is some minimizer of $\eqref{e:StokesGenMinProblemFctlSharp}-\eqref{e:StatNSSharpConstraintsWeak}$. Due to Lemma~\ref{l:PropMinimizerJ0N} we know that it holds
	$\left\|\nabla\b u\right\|_{\b L^2(\Omega)}\leq\frac{\mu}{2K_\Omega}$
	and hence we have
	$\left\|\nabla\b u_\epsilon\right\|_{\b L^2(\Omega)}<\delta+\left\|\nabla\b u\right\|_{\b L^2(\Omega)}<\frac{\mu}{2K_\Omega}+\frac{\mu}{2K_\Omega}=\frac{\mu}{K_\Omega}$
	and the statement follows from Lemma~\ref{l:EpsilonUniqueStateEquiForSmallU}.
\end{proof}

\subsection{Convergence of the optimality system}\label{s:ConvOptSys}
In Section \ref{s:PhaseFieldOptCond} we have derived a necessary optimality system for the phase field problem by geometric variations. The same has been done for the sharp interface problem in Section \ref{s:SharpOptCond}. In the previous subsection we have connected those two problems by showing that minimizers of the diffuse interface problem converge under certain assumptions to a minimizer of the sharp interface problem. We now complete this picture by showing that also the optimality conditions of the phase field problem can be shown to be an approximation of the derived necessary optimality system in the sharp interface setting. This is the content of the following theorem:

\begin{thm}\label{t:StokesConvOptSys}
Let $\left(\varphi_\epsilon,\b u_\epsilon\right)_{\epsilon>0}$ be the minimizers of the phase field problems $\eqref{e:StatNSGenMinProblemFctl}-\eqref{e:StatNSGenConstraintsWeak}$ as in Theorem~\ref{t:StatNSSharpConvergence}, thus it holds $\lim_{\epsilon\searrow0}\|\varphi_\epsilon-\varphi_0\|_{L^1(\Omega)}=0$ together with the convergence rate $\eqref{e:StatNsGammaLimitGrothcondOnMinim}$ and $\lim_{\epsilon\searrow0}\|\b u_\epsilon-\b u_0\|_{\b H^1(\Omega)}=0$. Then $\left(\varphi_0,\b u_0\right)$ is a minimizer of the sharp interface problem $\eqref{e:StokesGenMinProblemFctlSharp}-\eqref{e:StatNSSharpConstraintsWeak}$ and $\lim_{\epsilon\searrow0}J_\epsilon\left(\varphi_\epsilon,\b u_\epsilon\right)=J_0(\varphi_0,\b u_0)$. Moreover it holds
\begin{equation}\begin{split}\label{e:StokesGenFctDiffVarParResVariationInLimit}
		&\lim_{\epsilon\searrow0}\partial_t|_{t=0}j_\epsilon\left(\varphi_\epsilon\circ T_t^{-1}\right)=\partial_t|_{t=0}j_0\left(\varphi_0\circ T_t^{-1}\right)
\quad\forall  T\in{\mathcal T}_{ad}.	\end{split}\end{equation}
If $\left|\left\{\varphi_0=1\right\}\right|>0$ then we have additionally the following convergence results:
	\begin{subequations}\label{e:ConvOptSysAll}\begin{align}
	\lim_{\epsilon\searrow0}\lambda_\epsilon=\lambda_0,\qquad\lim_{\epsilon\searrow0}\left\|\dot{\b u}_\epsilon\left[V\right]-\dot{\b u}_0\left[V\right]\right\|_{\b H^1(\Omega)}=0 
	\end{align}\end{subequations}
	where $\{\b u_\epsilon\}=\b S_\epsilon(\varphi_\epsilon)$ for $\epsilon$ small enough and $\{\b u_0\}=\b S_0(\varphi_0)$. Moreover, $\left(\lambda_\epsilon\right)_{\epsilon>0}\subseteq\R^+_0$ are Lagrange multipliers for the integral constraint defined due to Theorem~\ref{l:StatNSDiffParVarConvOptCondDistProb}, $\lambda_0\geq0$ is a Lagrange multiplier such that it holds $\eqref{e:StatNSSharpParVarConvOptCondLimitDist}$, and thus is a Lagrange multiplier for the integral constraint in the sharp interface according to Theorem~\ref{t:StatNSSharpParVarConvOptCondDistProb}.
	
\end{thm}
\begin{remark}
The additional condition $\left|\left\{\varphi_0=1\right\}\right|>0$ is only necessary in order to obtain the convergence of the Lagrange multipliers. But as already discussed in \cite{GarckeHechtStokes}, this condition is not very restrictive.
\end{remark}

\begin{proof}
	We can apply the ideas of \cite[Theorem 5]{GarckeHechtStokes}. The nonlinearity can be included as in the proof of Lemma \ref{l:FBEpsilonSharpConv} and Lemma \ref{l:FBEpsilonSharpConvWithoutBound}. For details we refer to \cite[Section 17]{hecht}.
\end{proof}

\section{Concluding remarks}\label{s:Concluding}
Summarizing, we have shown that the phase field approach, which was proposed and discussed in Section~\ref{s:Phasefield}, approximates the sharp interface model $\eqref{e:StokesGenMinProblemFctlSharp}-\eqref{e:StatNSSharpConstraintsWeak}$ describing topology optimization problems in a stationary Navier-Stokes flow in a sharp interface setting, in the following sense: We know, that for any sequence of minimizers of the phase field problems, there exists a subsequence that converges to some limit element as the thickness of the interface tends to zero. If this sequence fulfills a certain convergence rate we find, that it actually converges in the strong $L^1(\Omega)\times\b H^1(\Omega)$ topology and that the limit element is a minimizer of the sharp interface model. Moreover, we can show in this setting that certain optimality conditions of the phase field model approximate an optimality system of the sharp interface model. As we have proven that those optimality conditions of the sharp interface are, under suitable assumptions, equivalent to classical shape derivatives, this gives that the optimality conditions of the phase field model are for small $\epsilon>0$ also an approximation of shape derivatives. This implies, that the phase field formulation is a good approximation for the shape topology optimization problem in a sharp interface setting and is consistent with existing models.\\

One can also include a pressure depending term in the objective functional, hence minimize
$$\int_\Omega f\left(x,\b u,\der\b u,p\right)\dx$$
if one includes the restriction that there is fluid on the parts of the domain where the pressure $p$ is taken into account. This is discussed in more detail in \cite[Section 6]{GarckeHechtStokes} for the Stokes equations but can also be applied directly to the stationary Navier-Stokes equations, see also \cite[Section 22]{hecht}.

\bibliographystyle{plain}
\bibliography{literature}

\end{document}